\documentclass{amsart}
\usepackage[utf8]{inputenc}
\usepackage{amsmath, amssymb, amsthm, bm}
\usepackage{tikz-cd}
\usepackage{tikz}
\usepackage[margin = 2cm]{geometry}
\usepackage[pdftex, colorlinks, linkcolor = blue, urlcolor = blue, citecolor = blue]{hyperref}

\newtheorem{theorem}{Theorem}[section]
\newtheorem{lemma}[theorem]{Lemma}

\newtheorem{corollary}[theorem]{Corollary}
\theoremstyle{definition}

\theoremstyle{definition}
\newtheorem{definition}[theorem]{Definition}
\theoremstyle{definition}
\newtheorem{remark}[theorem]{Remark}
\theoremstyle{remark}
\newtheorem*{claim}{Claim}
\numberwithin{equation}{section}

\newcommand{\C}{\mathbb{C}}
\newcommand{\Z}{\mathbb{Z}}
\newcommand{\N}{\mathbb{N}}
\newcommand{\g}{\mathfrak{g}}
\newcommand{\h}{\mathfrak{h}}
\newcommand{\n}{\mathfrak{n}}
\newcommand{\p}{\mathfrak{p}}
\newcommand{\fbar}{\overline{f}}
\newcommand{\calO}{\mathcal{O}}

\newcommand{\Ext}{\operatorname{Ext}}
\newcommand{\Hom}{\operatorname{Hom}}
\newcommand{\rank}{\operatorname{rank}}
\newcommand{\dimension}{\operatorname{dim}}
\newcommand{\vspan}{\operatorname{span}}
\newcommand{\im}{\operatorname{im}}
\newcommand{\id}{\operatorname{id}}

\newcommand{\supp}{\operatorname{supp}}
\newcommand{\wt}{\operatorname{wt}}
\newcommand{\ch}{\operatorname{ch}}
\newcommand{\ad}{\operatorname{ad}}

\title{Category $\calO$ for Takiff Lie algebras}
\author{Matthew Chaffe}
\address{School of Mathematics, University of Birmingham, Birmingham, B15 2TT, UK}
\email{mxc167@student.bham.ac.uk}
\date{}
\begin{document}
\maketitle

\begin{abstract}
We study category $\calO$ for Takiff Lie algebras $\g \otimes \C[\epsilon]/(\epsilon^2)$ where $\g$ is the Lie algebra of a reductive algebraic group over $\C$. We decompose this category as a direct sum of certain subcategories and use an analogue of parabolic induction functors and twisting functors for BGG category $\calO$ to prove equivalences between these subcategories. We then use these equivalences to compute the composition multiplicities of the simple modules in the Verma modules in terms of composition multiplicities in the BGG category $\calO$ for reductive subalgebras of $\g$. We conclude that the composition multiplicities are given in terms of the Kazhdan--Lusztig polynomials.
\end{abstract}

\section{Introduction}

\subsection{Category $\calO$}

For a reductive Lie algebra $\g$ over $\C$, an important part of the representation theory of $\g$ is the subcategory $\calO$ of $U(\g)$-Mod introduced by Bernstein--Gelfand--Gelfand (BGG) in the 1970s. This category contains many interesting modules, including all finite dimensional modules. The objects of this category are the modules satisfying certain finiteness conditions, and the category has many desirable homological properties. One of the main results in the theory of category $\calO$ was the proof of the Kazhdan--Luzstig conjecture, which gives the composition multiplicities of the simple modules in the Verma modules in terms of values of certain polynomials, called the Kazhdan--Luzstig polynomials, at 1. The definition of category $\calO$ was later extended to a more general class of Lie algebras, namely those with a triangular decomposition (see for example \cite{RCW}).

In this paper, we consider the Takiff Lie algebra of a reductive Lie algebra $\g$. Such algebras do not have a triangular decomposition in the sense of \cite{RCW}, but they do have a decomposition that is triangular in a weaker sense (see \cite[\S2]{W}) which allows us define an analogue of BGG category $\calO$. The case of Takiff $\mathfrak{sl}_2$ was studied by Mazorchuk and Söderberg in \cite{MS}. In particular, they gave a natural analogue of category $\calO$ for Takiff $\mathfrak{sl}_2$. We study the natural generalisation of this definition to Takiff $\g$ for a general reductive Lie algebra $\g$. Our main result is a formula for the composition multiplicities of the simple modules in the Verma modules in this category.

\subsection{Takiff Lie algebras and the category $\calO_\epsilon$}

The main type of Lie algebras we consider in this paper are the Takiff Lie algebras. For any Lie algebra $\g$, we define Takiff $\g$, denoted $\g_\epsilon$, to be the Lie algebra $\g_\epsilon = \g \otimes \C[\epsilon]/(\epsilon^2)$. We consider these algebras in the case where $\g$ is the Lie algebra of a reductive group. Such algebras were first considered by Takiff in \cite{T}, in which the invariant polynomials are considered. More recently, highest weight theory for these algebras was considered in \cite{W} and their graded representation theory was considered in \cite{CG}.

We write $x$ and $\overline{x}$ respectively for the elements $x \otimes 1$ and $x \otimes \epsilon \in \g_\epsilon$, and if $\mathfrak{a} \subseteq \g$, then we write $\mathfrak{a}$, $\overline{\mathfrak{a}}$, and $\mathfrak{a}_\epsilon$ for the subsets $\{a \in \g_\epsilon: a \in \mathfrak{a}\}$, $\{\overline{a} \in \g_\epsilon: a \in \mathfrak{a}\}$, and $\mathfrak{a} \otimes \C[\epsilon]/(\epsilon^2)$ of $\g_\epsilon$ respectively. If $u = \sum_i x_{i, 1} x_{i, 2} \dots x_{i, n_i} \in U(\g)$, the enveloping algebra of $\g$, for some $x_{i, j} \in \g$ and $n_i \in \Z_{>0}$ then we also write $\overline{u} = \sum_i \overline{x}_{i, 1} \overline{x}_{i, 2} \dots \overline{x}_{i, n_i} \in U(\g_\epsilon)$.

We fix a maximal torus $\h$ of $\g$, root system $\Phi$, and choice of positive roots $\Phi^+$, which gives a triangular decomposition $\n^- \oplus \h \oplus \n$ of $\g$. We then have a direct sum decomposition $\n^-_\epsilon \oplus \h_\epsilon \oplus \n_\epsilon$ of $\g_\epsilon$. This is not triangular in the sense of \cite{RCW} since $\h_\epsilon$ does not act diagonalisably on $\g_\epsilon$, but it is triangular in the weaker sense of \cite[\S2]{W}.

Highest weight theory for truncated current Lie algebras, of which Takiff algebras are a special case, was considered in \cite[\S3]{W}. In the case of Takiff algebras, the Verma modules are defined to be $M_{\lambda, \mu} = U(\g_\epsilon) \otimes_{U(\h_\epsilon \oplus \n_\epsilon)} \C_{\lambda, \mu}$ where $\C_{\lambda, \mu}$ is the 1-dimensional $U(\h_\epsilon \oplus \n_\epsilon)$-module on which $\h$ and $\overline{\h}$ act by $\lambda \in \h^*$ and $\mu \in \h^*$ respectively, and $\n_\epsilon$ acts by 0. We go further and define the category $\calO_\epsilon$ of $U(\g_\epsilon)$-modules to be the full subcategory of $U(\g_\epsilon)$-Mod with objects $M$ such that $M$ is finitely generated, the subalgebra $\h$ acts semisimply on $M$, and both $\n_\epsilon$ and $\overline{\h}$ act locally finitely on $M$. This definition is a generalisation of the definition of \emph{classical category $\calO$} for Takiff $\mathfrak{sl}_2$ given in \cite[\S2.4]{MS}. We will sometimes write expressions such as $\calO_\epsilon (\g)$ to emphasise to which algebra $\g$ we are referring. As in BGG category $\calO$, the Verma module $M_{\lambda, \mu}$ has a unique simple quotient $L_{\lambda, \mu}$ and these $L_{\lambda, \mu}$ form a set of representatives of the isomorphism classes of simple modules in $\calO_\epsilon$. In particular, both the Verma modules and simple modules in $\calO_\epsilon$ are parameterised by $\h^* \oplus \h^*$.

We now mention two differences between our category $\calO_\epsilon$ and the BGG category $\calO$. Firstly, the category $\calO_\epsilon$ is not Artinian; this will follow from our computation of composition multiplicities in \S6. Secondly, unlike in BGG category $\calO$ projective covers need not exist. In the $\g = \mathfrak{sl}_2$ case, for example, \cite[Theorem 17]{MS} states that $\calO_\epsilon$ decomposes as a direct sum of certain subcategories. Some of these subcategories are equivalent to the category of finite-dimensional modules for the power series ring $\C[[X]]$, in which it is easy to see that the trivial module has no projective cover.

\subsection{Composition multiplicities of Verma modules}

Although the modules in $\calO_\epsilon$ are not necessarily finite length, we can give a well defined notion of $[M: L_{\lambda, \mu}]$, the composition multiplicity of $L_{\lambda, \mu}$ in $M$ (see Lemmas \ref{compmultdefn} and \ref{compmultalternatedefn}). In BGG category $\calO$, the composition multiplicities are given by the values at 1 of certain Kazhdan--Lusztig polynomials (see for example \cite{HTT}). In the category $\calO_\epsilon$, we can calculate the composition multiplicities using the following result, where we write $p$ for Kostant's partition function (with the same sign convention as in \cite[\S1.16]{H}), and write $\bullet_2$ for the shifted action of $W$ on $\h^*$ given by $w \bullet_2 \lambda = w(\lambda + 2\rho) - 2\rho$ where $\rho = \sum_{\alpha \in \Phi^+} \frac{1}{2} \alpha$. Here we extend elements $\mu \in \h^*$ to elements of $\g^*$ by setting $\mu(x) = 0$ for $x \in \n \oplus \n^-$ and extending linearly, and we then write $\g^\mu$ to mean the centraliser (with respect to the coadjoint action) of this extension. We also recall that a standard parabolic subalgebra of $\g$ is a subalgebra containing the standard Borel $\h \oplus \n$.

\begin{theorem}
\label{maincompmultthm}
Let $\lambda, \lambda', \mu, \mu' \in \h^*$ and let $w$ be an element of the Weyl group $W$ of $\g$ of minimal length such that $\g^{w(\mu)}$ is the Levi factor of a standard parabolic subalgebra. Then:
\begin{align*}
[M_{\lambda, \mu}: L_{\lambda', \mu'}] = \delta_{\mu \mu'}\sum_{\chi \in \mathbb{Z}\Phi} p(\chi) [M_{w \bullet_2 \lambda + \chi}(\g^{w(\mu)}): L_{w \bullet_2 \lambda'}(\g^{w(\mu)})]
\end{align*}
where $M_{w \bullet_2 \lambda + \chi}(\g^{w(\mu)})$ and $L_{w \bullet_2 \lambda'}(\g^{w(\mu)})$ are respectively the Verma module and the simple module for $\g^{w(\mu)}$. All but finitely many terms of this sum are zero, so $[M_{\lambda, \mu}: L_{\lambda, \mu}]$ is always finite.
\end{theorem}

The composition multiplicities of the simple modules in the Verma modules are therefore given in terms of the composition multiplicities of the simple modules in the Verma modules for the reductive algebra $\g^\mu$, and hence in terms of the Kazhdan--Lusztig polynomials.

\subsection{Parabolic induction and twisting functors}

The category $\calO_\epsilon$ decomposes as a direct sum $\calO_\epsilon = \bigoplus_{\mu \in \h^*} \calO_\epsilon^\mu$, where $\calO_\epsilon^\mu$ is the full subcategory of $\calO_\epsilon$ with objects consisting of the modules $M$ such that for any $h \in \h$, the action of $\overline{h} - \mu(h)$ on $M$ is locally nilpotent (see Lemma \ref{calOdecomposition}). The subcategory $\calO^\mu_\epsilon$ contains $M_{\lambda, \mu}$ and $L_{\lambda, \mu}$ for any $\lambda \in \h^*$. To prove Theorem \ref{maincompmultthm}, we reduce to the case $\calO^0_\epsilon$ using two equivalences. The first uses an analogue of the concept of twisting functors for BGG category $\calO$ (see \cite{AS} for a discussion of these functors in the BGG category $\calO$ case).

\begin{theorem}
\label{twistingtheoremshort}
Let $\alpha$ be a simple root, let $s_\alpha \in W$ be the simple reflection corresponding to $\alpha$, and let $\mu \in \h^*$ be such that $\mu(h_\alpha) \neq 0$. Then the categories $\calO_\epsilon^\mu(\g)$ and $\calO_\epsilon^{s_\alpha(\mu)}(\g)$ are equivalent.
\end{theorem}

By standard results concerning the action of $W$ on $\h^*$, for any $\mu \in \h^*$ there exists $w \in W$ such that $\calO^\mu_\epsilon$ is equivalent to $\calO^{w(\mu)}_\epsilon$ and $\g^{w(\mu)}$ is the Levi factor of a standard parabolic subalgebra of $\g$. The second equivalence uses a form of parabolic induction to obtain the following result (see Theorem \ref{maintheorem} for a more precise statement).

\begin{theorem}
\label{maintheoremshort}
Let $\mu \in \h^*$ be such that the centraliser $\g^\mu$ is the Levi factor of a standard parabolic subalgebra. Then the categories $\calO_\epsilon^\mu(\g)$ and $\calO_\epsilon^\mu(\g^\mu)$ are equivalent.
\end{theorem}

We also show that there is an equivalence between $\calO_\epsilon^\mu(\g^\mu)$ and $\calO_\epsilon^0(\g^\mu)$. Combining this with Theorems \ref{twistingtheoremshort} and \ref{maintheoremshort}, we see that any $\calO_\epsilon^\mu(\g)$ is equivalent to $\calO_\epsilon^0(\g^{w(\mu)})$ for a suitable element $w$ of the Weyl group of $\g$. Therefore, once we understand the image of the simple modules and Verma modules under these equivalences, the problem of computing composition multiplicities of the simple modules in the Verma modules in general reduces to computing them in $\calO_\epsilon^0(\g)$ for all reductive Lie algebras $\g$.

To compute the composition multiplicities of the Verma modules in $\calO_\epsilon^0(\g)$, we decompose them as a direct sum of $\h$-modules, each of which is isomorphic as an $\h$-module to a Verma module for $\g$. Since the simple modules in $\calO_\epsilon^0$ and the simple modules in BGG category $\calO$ are isomorphic as $\g$-modules (see Lemma \ref{simple0modules}), this gives Theorem \ref{maincompmultthm} in the case $\mu = 0$. Applying Theorems \ref{twistingtheoremshort} and \ref{maintheoremshort} then gives the full result.

\subsection{Structure of the paper}

In \S2.1 and \S2.2, we define our basic notation, and in \S2.3 we review some standard results on centralisers, Levi subalgebras, and parabolic subalgebras.

In \S3.1 we give analogues of several fundamental results in BGG category $\calO$ that hold in $\calO_\epsilon$, and in \S3.2 we prove the existence of the decomposition $\calO_\epsilon = \bigoplus_{\mu \in \h^*} \calO_\epsilon^\mu$ mentioned above.

In \S4 we state precisely and prove Theorem \ref{maintheoremshort}. In \S4.1 we prove a result on the action of the centre $Z(\g_\epsilon)$ of the enveloping algebra of $\g_\epsilon$ on highest weight modules, and then in \S4.2 use this result to show the exactness the functors in the equivalence. We complete the proof using a standard argument involving the Five Lemma inspired by \cite[Proposition 2.1]{FP}.

Twisting functors for $\calO_\epsilon$ are defined in \S5.1 and shown to be equivalences between certain $\calO_\epsilon^\mu$ in \S5.2. 

Finally, in \S6 we first show that the notion of composition multiplicity is well defined in our category and then proceed to compute the composition multiplicities of the simple modules in the Verma module $M_{\lambda, 0}$. By the parabolic induction and twisting functor equivalences, this is then enough to determine the composition multiplicities for all Verma modules in $\calO_\epsilon$.

\subsection*{Acknowledgements}
The author would like to thank Simon Goodwin and Lewis Topley for their support and advice, and the EPSRC for financial support.

\section{Preliminaries}

\subsection{Notation and conventions}

Throughout, all vector spaces, Lie algebras, associative algebras are over $\C$, and unless otherwise specified tensor products are over $\C$. Associative algebras are unital and not necessarily commutative. Unless otherwise stated, if $A$ is an associative $\C$-algebra, then by an $A$-module we always mean a left $A$-module. 

\subsection{Reductive groups and Lie algebras}

Let $\g$ be the Lie algebra of a connected reductive algebraic group $G$ over $\C$. Choose a maximal torus $\h$ of $\g$, and let $\Phi \subseteq \h^*$ be the root system of $\g$ with respect to $\h$, writing $\g_\alpha$ for the $\alpha$ root space of $\g$. We fix a system of simple roots $\Delta \subseteq \Phi$ and corresponding system of positive roots $\Phi^+ \subseteq \Phi$. We fix a partial order on $\Phi$ in the usual way. We fix a basis $\{e_\alpha : \alpha \in \Phi^+\} \cup \{f_\alpha : \alpha \in \Phi^+\} \cup \{h_\alpha : \alpha \in \Delta\}$ of $\g$ such that $e_\alpha \in \g_\alpha$, $f_\alpha \in \g_{-\alpha}$, and $h_\alpha := [e_\alpha, f_\alpha] \in \h$ is such that $\alpha(h_\alpha) = 2$. We will sometimes write $e_{-\alpha}$ instead of $f_\alpha$. For $\alpha, \beta \in \Phi$ such that $\alpha + \beta \in \Phi$, let $\kappa_{(\alpha, \beta)} \in \C$ be such $[e_\alpha, e_\beta] = \kappa_{(\alpha, \beta)} e_{\alpha + \beta}$. Let $\n = \vspan\{e_\alpha : \alpha \in \Phi^+\}$, let $\n^- = \vspan\{f_\alpha : \alpha \in \Phi^+\}$, and let $\mathfrak{b} = \h \oplus \n$ be the standard Borel subalgebra of $\g$. Let $W$ be the Weyl group of $\g$, and for $\alpha \in \Delta$ let $s_\alpha \in W$ be the simple reflection corresponding to $\alpha$. If $H \subseteq G$ is a maximal toral subgroup of $G$ corresponding to $\h$, then $W$ can also be viewed as $N_G(H)/H$, where $N_G(H)$ is the normaliser of $H$ in $G$. We write $\g_\epsilon$ for Takiff $\g$ and write $x$ for $x \otimes 1 \in \g_\epsilon$ and $\overline{x}$ for $x \otimes \epsilon \in \g_\epsilon$. There is a triangular decomposition of $\g$ given by $\g_\epsilon = \n^-_\epsilon \oplus \h_\epsilon \oplus \n_\epsilon$.

\subsection{Centralisers and parabolic subalgebras}

Let $\mu \in \h^*$, and extend $\mu$ to an element of $\g^*$ by setting $\mu(x) = 0$ for $x \in \n \oplus \n^-$. The centraliser $\g^\mu$ can be described as $\g^\mu = \h \oplus (\bigoplus_{\beta \in \Phi_\mu} \g_\beta$), where $\Phi_\mu = \{\beta \in \Phi : \mu(h_\beta) = 0\}$. This is the Levi factor of some parabolic subalgebra $\mathfrak{q}$ of $\g$, which by \cite[Lemma 3.8.1]{CM} is conjugate to a standard parabolic subalgebra $\p$. It then follows that $\g^\mu$ is conjugate to the Levi factor $\mathfrak{l}$ of $\p$, and then (again using \cite[Lemma 3.8.1]{CM}) the root system $\Phi_\mu$ is $W$-conjugate to the root system of $\mathfrak{l}$ by some element $w \in W$. We then observe that $\Phi_{w(\mu)} = w(\Phi_\mu)$, and so $\mathfrak{l} = \g^{w(\mu)}$. This gives us the first part of the following lemma, which will allow us to use Theorem \ref{twistingtheoremshort} to show that any $\calO_\epsilon^\mu$ is equivalent to some $\calO_\epsilon^{\mu'}$, where $\mu'$ satisfies the hypotheses of Theorem \ref{maintheoremshort}:

\begin{lemma}
Let $\mu \in \h^*$. Then there exists $w \in W$ and a Levi factor $\mathfrak{l}$ of a standard parabolic $\p$ such that $\g^{w(\mu)} = \mathfrak{l}$. Furthermore, if we pick $w$ of minimal length subject to this condition and let  $w = s_{\alpha_n} s_{\alpha_{n-1}} \dots s_{\alpha_1}$ be a reduced expression for $w$, then for each $1 \leq i \leq n$, we have $((s_{\alpha_{i-1}} \dots s_{\alpha_1})\mu)(h_{\alpha_i}) \neq 0$.
\end{lemma}

\begin{proof} By the above discussion, we can certainly find $w = s_{\alpha_n} s_{\alpha_{n-1}} \dots s_{\alpha_1}$ and $\mathfrak{l}$ the Levi factor of a standard parabolic subalgebra satisfying $\g^{w(\mu)} = \mathfrak{l}$. Now, suppose that $w$ has minimal length such that $g^{w(\mu)} = \mathfrak{l}$ and that for some $i$ we have $((s_{\alpha_{i-1}} \cdots s_{\alpha_1})\mu)(h_{\alpha_i}) = 0$. In general, if $\mu(h_\alpha) = 0$ for some $\alpha \in \Phi$ then $s_\alpha(\mu) = \mu - \mu(h_\alpha) \alpha = \mu$, so in particular, if $w' := s_{\alpha_n} \dots s_{\alpha_{i+1}} s_{\alpha_{i-1}} \dots s_{\alpha_1}$, then $w(\mu) = w'(\mu)$ and so $\mathfrak{l} = \g^{w(\mu)} = \g^{w'(\mu)}$. But $w'$ has shorter length than $w$, giving a contradiction.
\end{proof}

\section{Category $\calO$ for Takiff Lie algebras}

\subsection{Elementary results in $\calO_\epsilon$}

We now formally state the definition of the category $\calO_\epsilon$ described in the introduction:

\begin{definition}
\label{catOdef}
The category $\calO_{\epsilon}$ is the full subcategory of $U(\g_\epsilon)$-Mod with objects $M$ satisfying the following:

\begin{enumerate}
\item[$(\calO1)$] $M$ is finitely generated.

\item[$(\calO2)$] $\h$ acts semisimply on $M$.

\item[$(\calO3)$] $\n_\epsilon$ and $\overline{\h}$ act locally finitely on $M$.
\end{enumerate}
\end{definition}

As for BGG category $\calO$, the category $\calO_\epsilon$ is closed under submodules, quotients, and direct sums, and every module $M \in \calO_\epsilon$ is Noetherian.

Let $M \in \calO_{\epsilon}$. We say $v \in M$ is a \emph{weight vector} of weight $\lambda \in \h^*$ if for any $h \in \h$, we have $h \cdot v = \lambda(h)v$, and write $M^\lambda$ for the subspace of $M$ consisting of weight vectors of weight $\lambda$. We say $v \in M$ is a \emph{highest weight vector} of weight $(\lambda, \mu) \in \h^* \oplus \h^*$ if it satisfies the following:

\begin{enumerate}
\item[(1)] $\n_\epsilon \cdot m = 0$.

\item[(2)] For all $h \in \h$, we have $h \cdot m = \lambda(h)m$, i.e. $m \in M^\lambda$.

\item[(3)] For all $h \in \h$, we have $\overline{h} \cdot m = \mu(h)m$.
\end{enumerate}
We also say $m$ is \emph{maximal} of weight $\lambda$ if it satisfies conditions (1) and (2) but not necessarily (3).

\begin{lemma}
\label{fdweightspaces}
Let $M \in \calO_{\epsilon}$. Then:

\begin{enumerate}
\item[(a)] Each weight space $M^\lambda$ of $M$ is finite dimensional; and

\item[(b)] The set $\{\lambda \in \h^* : M^{\lambda} \neq 0\}$  is contained in $\bigcup_{\lambda \in I}\{\lambda - \gamma: \gamma \in \Z_{\geq 0}\Phi^+\}$ for some finite subset $I \subseteq \h^*$.
\end{enumerate}
\end{lemma}

\begin{proof} The proofs are a natural extension of the proofs of the equivalent statements in the BGG category $\calO$ (see the proofs of ($\calO4$) and ($\calO5$) in \cite[\S1.1]{H}).
\end{proof}

\begin{corollary}
\label{maxhwcorollary}
Let $M \in \calO_\epsilon$ and suppose there exists a maximal vector of weight $\lambda$ in $M$. Then for some $\mu \in \h^*$ there exists a highest weight vector of weight $(\lambda, \mu)$ in $M$.  
\end{corollary}

\begin{proof} Let $V$ be the space of maximal vectors of weight $\lambda$. This is finite dimensional since it is a subspace of $M^\lambda$, which is finite dimensional by Lemma \ref{fdweightspaces}. Now, the action of $\overline{\h}$ preserves weight spaces and for any $h \in \h$ and $n \in \n$ or $\overline{\n}$ we have $n \cdot (\overline{h} \cdot v) = \overline{h} \cdot (n \cdot v) + [n, \overline{h}] \cdot v$. Hence since $[n, \overline{h}]$ lies in $\overline{\n}$, the action of $\overline{\h}$ preserves maximal vectors, so $\overline{\h}$ acts on $V$, so since $\overline{\h}$ is commutative there is some common eigenvector $v \in V$ for this action. But then by definition $v$ is a highest weight vector of weight $(\lambda, \mu)$ for some $\mu \in \h^*$.
\end{proof}

We say $M$ is a highest weight module of weight $(\lambda, \mu)$ if there is some highest weight vector $v \in M$ of weight $(\lambda, \mu)$ that generates $M$. One class of highest weight modules are the \emph{Verma modules} $M_{\lambda, \mu}$, which are defined by $M_{\lambda, \mu} = U(\g_\epsilon) \otimes_{U(\mathfrak{b}_\epsilon)} \C_{\lambda, \mu}$ where $\C_{\lambda, \mu}$ is the one dimensional $U(\mathfrak{b}_\epsilon)$-module where $\n_\epsilon$ acts by 0, $\h$ acts by $\lambda$, and $\overline{\h}$ acts by $\mu$. These Verma modules are the universal highest weight modules in the sense that for any highest weight module $M$ of weight $(\lambda, \mu)$, we have that $M_{\lambda, \mu}$ maps onto $M$ (by sending $1 \otimes 1_{\lambda, \mu} \in M_{\lambda, \mu}$ to a non-zero highest weight vector $v \in M$ of weight $(\lambda, \mu)$ that generates $M$), and this map is unique up to scalar multiplication. They behave similarly to the Verma modules in BGG category $\calO$; for example by a similar argument to \cite[Theorem 1.2]{H} for any Verma module $M_{\lambda, \mu}$ the weight space $M_{\lambda, \mu}^\lambda$ is 1 dimensional, and $M_{\lambda, \mu}$ has a unique maximal proper submodule and hence a unique simple quotient denoted $L_{\lambda, \mu}$. Once we have Lemma \ref{finitefiltration} it will easily follow that these $L_{\lambda, \mu}$ form an irredundant set of representatives of isomorphism classes of simple modules in $\calO_\epsilon$. We also observe that any finite dimensional simple module must lie in $\calO_\epsilon$ and so must be equal to $L_{\lambda, \mu}$ for some $\lambda, \mu \in \h^*$. We now prove the following lemma, which is an analogue of a standard result (see \cite[Corollary 1.2]{H}) in BGG category $\calO$.

\begin{lemma}
\label{finitefiltration}
Let $M \in \calO_{\epsilon}$. Then $M$ has a finite filtration 
\[0 = M_0 \subseteq M_1 \subseteq M_2 \subseteq \cdots \subseteq M_{k-1} \subseteq M_k = M\]
such that each $M_{i+1}/M_i$ is a highest weight module.
\end{lemma}

\begin{proof}
By $(\calO1)$ and $(\calO2)$, there exists a finite set $\{v_1, v_2, \dots , v_n\}$ of weight vectors which generate $M$. Let $V$ be the $U(\n_\epsilon \oplus \overline{\h})$-module generated by the $v_i$, which is finite dimensional by $(\calO3)$, and proceed by induction on $\dimension(V)$, observing that if $m$ is maximal, then so are $h \cdot m$ and $\overline{h} \cdot m$ for any $h \in \h$:

If dim($V$) = 1, then any non-zero element of $V$ is a highest weight vector and generates $M$, so $M$ is highest weight. If dim($V$) $>$ 1 then pick some $\lambda$ maximal among the weights of $V$. Then any vector of weight $\lambda$ is maximal, so by a similar argument to the proof of Corollary \ref{maxhwcorollary} there is some $v \in V^{\lambda}$ which is a highest weight vector generating a highest weight submodule $M_1$ of $M$. Hence we may consider $\overline{M} = M/M_1$ which is generated by $\overline{V}$, the image of $V$ in $M/M_1$. Since dim($\overline{V}$) $<$ dim($V$) we are done by induction.
\end{proof}

\begin{corollary}
The set $\{L_{\lambda, \mu}: \lambda, \mu \in \h^*\}$ is an irredundant set of representatives of isomorphism classes of simple modules in $\calO_\epsilon$.
\end{corollary}

\begin{proof}
Let $L$ be a simple module in $\calO_\epsilon$. By Lemma \ref{finitefiltration}, $L$ has a finite filtration such that each quotient is a highest weight module, but since $L$ is simple this filtration must have length 1, i.e. $L$ is a highest weight module. Hence $L$ is a simple quotient of some Verma module $M_{\lambda, \mu}$, but the only such simple quotient is $L_{\lambda, \mu}$.
\end{proof}

\begin{lemma}
\label{centreactionlemma}
The centre $Z(\g_\epsilon)$ of $U(\g_\epsilon)$ acts on any highest weight module $M$ by some character $\chi: Z(\g_\epsilon) \rightarrow \C$.
\end{lemma}

\begin{proof}
Again the proof is a natural extension of the proof of the equivalent statement in BGG category $\calO$.
\end{proof}

\subsection{Decomposition of $\calO_{\epsilon}$}

We now wish to decompose $\calO_{\epsilon}$ into a direct sum of smaller subcategories. In BGG category $\calO$, for a central character $\chi$ the subcategory $\calO_\chi$ is defined to be the full subcategory of $\calO$ whose objects are the modules with generalised central character $\chi$. It can then be shown that $\calO = \bigoplus_\chi \calO_\chi$ and hence any indecomposable module lies in $\calO_\chi$ for some $\chi$ (see \cite[\S1.12]{H} for more details). It is possible to construct a similar decomposition for our $\calO_\epsilon$ using Lemmas \ref{finitefiltration} and \ref{centreactionlemma} which are analogous to the results used in BGG category $\calO$. However, we choose to use the subalgebra $U(\overline{\h})$ instead of $Z(\g_\epsilon)$ via the following lemma (though we will return to the action of $Z(\g_\epsilon)$ in \S4.1):

\begin{lemma}
\label{generalisedeigenvalues}
For $M \in \calO_\epsilon$, let $M^{(-, \mu)}$ be the generalised weight space for the action of $\overline{\h}$ with weight $\mu$, i.e. the set of elements of $M$ such that for all $h \in \h$, there exists $n \in \mathbb{N}$ with $(h - \mu(h))^n \cdot m = 0$. Then:

\begin{enumerate}
\item[(a)] As a vector space, $M = \bigoplus_{\mu \in \h^*} M^{(-, \mu)}$.

\item[(b)] Each $M^{(-, \mu)}$ is a submodule of $M$.
\end{enumerate}
\end{lemma}

\begin{proof}
First observe that $\overline{\h}$ acts on each weight space $M^\lambda$, and since these are finite dimensional and $\overline{\h}$ is abelian, $M^\lambda$ may be written as the sum of generalised weight vectors for $\overline{\h}$.
Now let $v \in M$. Then $v$ can be written as a sum of weight vectors, which can each be written as a sum of generalised weight vectors for $\overline{\h}$. Hence $M$ is the sum of the the $M^{(-, \mu)}$ and this sum is direct, proving (a).

To show that $M^{(-, \mu)}$ is a submodule of $M$, it suffices to show that $x \cdot v \in M^{(-, \mu)}$ for any basis vector $x$ of $\g_\epsilon$. Now let $v \in M^{(-, \mu)}$. If $x = \overline{e_\alpha}, \overline{f}_\alpha, \overline{h}_\alpha$ or $h_\alpha$, then $[x, (\overline{h} - \mu(h))^n] = 0$ for any $h \in \h$ and $n \geq 0$, so $(\overline{h} - \mu(h))^n \cdot (x \cdot v) = x \cdot ((\overline{h} - \mu(h))^n \cdot v)$. But since $v \in M^{(-, \mu)}$, this is 0 for sufficiently large $n$ and hence $x \cdot v \in M^{(-, \mu)}$. For $x = e_\alpha$, we have $[e_\alpha, (\overline{h} - \mu(h))^n] = - n\alpha(h)\overline{e_\alpha}(\overline{h} - \mu(h))^{n-1}$, so:
\[(\overline{h} - \mu(h))^n \cdot (e_\alpha \cdot v) = e_\alpha \cdot ((\overline{h} - \mu(h))^n \cdot v) + n\alpha(h)\overline{e_\alpha}(\overline{h} - \mu(h))^{n-1} \cdot v\]
and again, since $v \in M^{(-, \mu)}$, this is 0 for sufficiently large $n$. Hence $e_\alpha \cdot v \in M^{(-, \mu)}$, and the argument to show $f_\alpha \cdot v \in M^{(-, \mu)}$ is similar.
\end{proof}

In light of this, we define $\calO_\epsilon^\mu$ to be the full subcategory of $\calO_\epsilon$ whose objects are the modules $M$ such that $M = M^{(-, \mu)}$. We then have:
\begin{corollary}
\label{calOdecomposition}
There is a direct sum decomposition $\calO_\epsilon = \bigoplus_{\mu \in \h^*} \calO_\epsilon^\mu$.
\end{corollary}

\begin{lemma}
\label{simple0modules}
Let $\mu \in \h^*$. Then $M_{\lambda, \mu}$ and $L_{\lambda, \mu} \in \calO^\mu_\epsilon$ for all $\lambda \in \h^*$. Moreover, if $I: \calO(\g) \rightarrow \calO^0_\epsilon(\g)$ is the functor induced by the surjective homomorphism $\g_\epsilon \rightarrow \g$ with kernel $\overline{\g}$, then $L_{\lambda, 0} \cong I(L_\lambda)$ for all $\lambda \in \h^*$. 
\end{lemma}

\begin{proof}
To see that $M_{\lambda, \mu} \in \calO^\mu_\epsilon$, observe that for any $h \in \h$, $(\overline{h} - \mu(h)) \cdot (1 \otimes 1_{\lambda, \mu}) = 0$, so $M_{\lambda, \mu}^{(-, \mu)} \neq 0$. But $M_{\lambda, \mu}$ is indecomposable, so $M_{\lambda, \mu} = M_{\lambda, \mu}^{(-, \mu)}$, i.e. $M_{\lambda, \mu} \in \calO^\mu_\epsilon$. Since $\calO^\mu_\epsilon$ is closed under taking quotients, we also see that $L_{\lambda, \mu} \in \calO^\mu_\epsilon$. Finally, $I(L_\lambda)$ is certainly simple and lies in $\calO_\epsilon^0$, so is isomorphic to $L_{\lambda', 0}$ for some $\lambda' \in \h^*$. But by considering weight spaces, we must have $\lambda' = \lambda$, so $I(L_\lambda) \cong L_{\lambda, 0}$.
\end{proof}

\begin{remark}
\label{vermaextensionsremark}
It follows from this and the fact that Verma modules are indecomposable that $L_{\lambda', \mu'}$ cannot occur as a subquotient of $M_{\lambda, \mu}$ unless $\mu = \mu'$.
\end{remark}

\section{Parabolic Induction}

\subsection{The centre of $U(\g_\epsilon)$}

We now give a precise statement of Theorem \ref{maintheoremshort}:

\begin{theorem}
\label{maintheorem}
Let $\g$ be a reductive Lie algebra with maximal toral subalgebra $\h$ as above and let $\p \subseteq \g$ be a standard parabolic subalgebra of $\g$ with Levi decomposition $\p = \mathfrak{l} \oplus \mathfrak{r}$. Let $\mu \in \h^*$ be such that $\g^\mu = \mathfrak{l}$. Then there is a category equivalence between $\calO_\epsilon^\mu(\g^\mu)$ and $\calO_\epsilon^\mu(\g)$ given by the functors:
\begin{align*}
I : \calO_\epsilon^\mu (\g^\mu) &\longrightarrow \calO_\epsilon^\mu (\g) \\
M &\longmapsto U(\g_\epsilon) \otimes_{U(\p_\epsilon)} M\\
R : \calO_\epsilon^\mu (\g) &\longrightarrow 
\calO_\epsilon^\mu (\g^\mu) \\
M &\longmapsto M^{\mathfrak{r}_\epsilon}
\end{align*}
where for the first functor, we inflate a $U(\g^\mu_\epsilon)$-module $M$ into a $U(\p_\epsilon)$-module by letting $U(\mathfrak{r}_\epsilon)$ act by 0. For the second functor, $M^{\mathfrak{r}_\epsilon} = \{m \in M : \mathfrak{r}_\epsilon \cdot m = 0\}$.
\end{theorem}

These functors are adjoint since we have inverse isomorphisms $\theta: \Hom(M, N^{\mathfrak{r}_\epsilon}) \rightarrow \Hom(U(\g_\epsilon) \otimes_{U(\p_\epsilon)} M, N)$ and $\eta: \Hom(U(\g_\epsilon) \otimes_{U(\p_\epsilon)} M, N) \rightarrow \Hom(M, N^{\mathfrak{r}_\epsilon})$ given by $\theta(f)(u \otimes m) = u \cdot f(m)$ and $\eta(g)(n) = g(1 \otimes n)$ respectively.

\begin{remark}
Let $\C_{0, \mu}$ be the one-dimensional $U(\g_\epsilon^\mu)$-module where $[\g^\mu, \g^\mu]_\epsilon$ and $\mathfrak{z}(\g^\mu)$ act by 0, and $\overline{\mathfrak{z}(\g^\mu)}$ acts by $\mu$. Then, since $\mu$ is 0 on $[\g^\mu, \g^\mu] \cap \h$, the functors $(-) \otimes_{U(\g_\epsilon^\mu)} \C_{0, \mu}$ and $(-) \otimes_{U(\g_\epsilon^\mu)} \C_{0, -\mu}$ give a pair of mutually inverse equivalences of categories between $\calO_\epsilon^\mu(\g^\mu)$ and $\calO_\epsilon^0(\g^\mu)$. Combining this with Theorem \ref{maintheorem} then gives an equivalence between $\calO_\epsilon^\mu(\g)$ and $\calO_\epsilon^0(\g^\mu)$.
\end{remark}

To prove Theorem \ref{maintheorem} we use an approach similar to that of \cite[Theorem 2.1]{FP} to show that the maps $\phi_M : M \longrightarrow (U(\g_\epsilon) \otimes_{U(\p_\epsilon)} M)^{\mathfrak{r}_\epsilon}$ given by $\phi_M(m) = 1 \otimes m$ and $\varphi_N : U(\g_\epsilon) \otimes_{U(\p_\epsilon)} N^{\mathfrak{r}_\epsilon} \longrightarrow N$ given by $\varphi_N(u \otimes n) = u \cdot n$ are isomorphisms. We first show that the functors $I$ and $R$ are exact, and then use the Five Lemma and Lemma \ref{finitefiltration} to reduce the proof that $\phi_M$ and $\varphi_N$ are always isomorphisms to the case where $M$ and $N$ are highest weight modules. The exactness of $I$ follows easily from the fact that $U(\g_\epsilon)$ is a free $U(\p_\epsilon)$-module, but showing exactness of $R$ is more difficult. The key result we need to prove the exactness of $R$ is the following:

\begin{theorem}
\label{vermaexttheorem}
Let $\lambda, \lambda', \mu, \mu' \in \h^*$ and let $N_1, N_2 \in \calO_\epsilon$ be highest weight modules of weights $(\lambda, \mu)$ and $(\lambda', \mu')$ respectively. Then $\Ext_{\calO_\epsilon}(N_1,N_2) = 0$ unless $\mu = \mu'$ and $\lambda - \lambda' \in \Z\Phi_\mu$, where $\Phi_\mu = \{\alpha \in \Phi : \mu(h_\alpha) = 0\}$. In particular, let $M \in \calO_\epsilon$ be indecomposable, let $0 \subseteq M_1 \subseteq \dots \subseteq M_n = M$ be a filtration such that all $M_i/M_{i-1}$ are highest weight, and let $(\lambda_i, \mu_i)$ be the weight of the highest weight module $M_i/M_{i-1}$. Then there exist $\lambda, \mu \in \h^*$ such that $\lambda_i \in \lambda + \Z\Phi_\mu$ and $\mu_i = \mu$ for all $i$.
\end{theorem}

We devote the rest of this section to the proof of this theorem in the case where $\mu$ satisfies the hypotheses of Theorem \ref{maintheorem}. This is the only case we require to prove Theorem \ref{maintheorem}; however Theorem \ref{vermaexttheorem} will in fact hold for any $\mu$ as a consequence of Theorem \ref{twistingtheoremshort}. We first recall that if $\mu \neq \mu'$ then by Corollary \ref{calOdecomposition}, we have that $M$ and $N$ lie in different $\calO_\epsilon^\mu$, so we restrict our attention to the case $\mu = \mu'$. We wish to use the central characters of the Verma modules $M_{\lambda, \mu}$. To do this, we first need some information about the centre $Z(\g_\epsilon)$ of $U(\g_\epsilon)$. The following results, respectively \cite[Theorem 4.1]{T} and \cite[Corollary 2.4.11]{D}, can be used to prove certain elements of $U(\g_\epsilon)$ lie in the centre $Z(\g_\epsilon)$:

\begin{theorem}
\label{takiffcentre}
Let $\{z_1, z_2, \dots , z_n\}$ be a set of algebraically independent homogeneous generators of $S(\g)^{\g}$. Define maps $\iota, D: S(\g) \rightarrow S(\g_\epsilon)$ by letting $\iota$ be the inclusion $S(\g) \hookrightarrow S(\overline{\g}) \subseteq S(\g_\epsilon)$ and letting
\[D(z) = \sum_{x_i \in B} \frac{\partial \, \overline{z}}{\partial \, \overline{x_i}} x_i\]
where $B$ is a basis for $\g$.

Then $\{\iota(z_1), \iota(z_2), \dots , \iota(z_n), D(z_1), D(z_2), \dots , D(z_n)\}$ is a set of algebraically independent generators for $S(\g_\epsilon)^{\g_\epsilon}$.
\end{theorem}

We will always apply this Lemma with $B = \{e_\alpha, f_\alpha : \alpha \in \Phi^+\} \cup \{h_\alpha : \alpha \in \Delta\}$

\begin{lemma}
Let $\mathfrak{a}$ be a finite dimensional Lie algebra and let $\omega: S(\mathfrak{a}) \rightarrow U(\mathfrak{a})$ be given by $\omega(x_1 x_2 \dots x_k) = \frac{1}{k!} \sum_{\sigma \in S_k} x_{\sigma(1)} x_{\sigma(2)} \dots x_{\sigma(k)}$ for $x_1, x_2, \dots x_k \in \mathfrak{a}$. Then $\omega$ restricts to a bijection between $S(\mathfrak{a})^{\mathfrak{a}}$ and $Z(\mathfrak{a})$.
\end{lemma}

The map $\omega$ is not an algebra isomorphism, but (using the notation of Theorem \ref{takiffcentre}) we do have that $\omega(D(z_i)) \in Z(\g_\epsilon)$ for each $1 \leq i \leq n$.
We wish to understand how these elements act on a highest weight module $M$ of weight $(\lambda, \mu)$, and in particular how they act on a highest weight vector of weight $(\lambda, \mu)$.

We now let $\pi : S(\g)^G \rightarrow S(\h)^W$ be the Chevalley restriction map, and define another map $\pi': U(\g_\epsilon)^{\h} \rightarrow S(\h_\epsilon)$ whose restriction to $Z(\g_\epsilon)$ is analogous to the Chevalley restriction map in the following manner. By the PBW theorem, there is a decomposition $U(\g_\epsilon)^{\h} = U(\h_\epsilon) \oplus I$, where $I = U(\g_\epsilon)^{\h} \cap (U(\g_\epsilon) \n_\epsilon)$. This $I$ is clearly a left ideal of $U(\g_\epsilon)^{\h}$. To see it is a right ideal, we observe that $I$ is also equal to $U(\g_\epsilon)^{\h} \cap (\n_\epsilon^- U(\g_\epsilon))$. Let $\pi': U(\g_\epsilon)^{\h} \rightarrow U(\h_\epsilon) = S(\h_\epsilon)$ be the projection along this decomposition, which is an algebra homomorphism since $I$ is a two sided ideal. For any $z \in Z(\g_\epsilon) \subseteq U(\g_\epsilon)^{\h}$, we have that $z$ acts on a highest weight vector of weight $(\lambda, \mu)$, and hence on all of $M_{\lambda, \mu}$, by $(\pi'(z))(\lambda, \mu)$ where $S(\h_\epsilon)$ is identified with $\C[\h_\epsilon^*]$.

For $\mu \in \h^*$, we define two maps $\xi_\mu$ and $\psi_\mu$ by:
\begin{align*}
\xi_\mu : \h^*&\rightarrow \C^n \\
\lambda &\mapsto (\pi'\omega D(z_i)(\lambda, \mu))_{i = 1, \dots , n} \\
\psi_\mu : \h^* &\rightarrow \C^n \\
\lambda &\mapsto (\omega D\pi(z_i)(\lambda, \mu))_{i = 1, \dots , n}
\end{align*}
We also let $\phi: \h^* \rightarrow \h^*/W \cong \C^n$ be the quotient map, which is the map of algebraic varieties induced by the inclusion $S(\h)^W \hookrightarrow S(\h)$. Here the identification $\h^*/W \cong \C^n$ is induced by the Chevally isomorphism $S(\h)^W \cong \C[\pi(z_1), \pi(z_2), \dots, \pi(z_n)]$. We now state and prove the following Lemma, which we use together with Corollary \ref{calOdecomposition} to prove Theorem \ref{vermaexttheorem}:

\begin{lemma}
\label{centralcharaterlemma}
With notation as above:

\begin{enumerate}
\item[(a)] Let $M$ and $N$ be highest weight modules of weight $(\lambda, \mu)$ and $(\lambda', \mu)$ respectively. Then $M$ and $N$ have different central characters unless $\xi_\mu(\lambda) = \xi_\mu(\lambda')$. 

\item[(b)] For any fixed $\mu \in \h^*$, there is some $\mathbf{c}_\mu \in \C^n$ such that $\xi_\mu(\lambda) = \psi_\mu(\lambda) + \mathbf{c}_\mu$. Hence $\xi_\mu(\lambda) = \xi_\mu(\lambda')$ if and only if $\psi_\mu(\lambda) = \psi_\mu(\lambda')$.

\item[(c)] $d_\mu \phi = \psi_\mu$ for any $\mu \in \h^*$.

\item[(d)] $\rank(d_\mu \phi) = \dimension(\mathfrak{z}(\g^\mu))$.

\item[(e)] If $\g^\mu$ is a Levi factor of a standard parabolic subalgebra $\p \subseteq \g$, we have $\ker(\psi_\mu) = \ker(d_\mu \phi) = \C\Phi_\mu$, so by part (b), $\xi_\mu(\lambda) = \xi_\mu(\lambda')$ holds only if $\lambda - \lambda' \in \C\Phi_\mu$.
\end{enumerate}
\end{lemma}

\begin{proof}
Part (a) follows from the remark that $z \in Z(\g_\epsilon)$ acts on $M_{\lambda, \mu}$ by $(\pi'(z))(\lambda, \mu)$.

For part (b), we first view $\pi' \circ \omega D$ and $\omega D \circ \pi$ as maps from $S(\g)^\h$ to $S(\h_\epsilon)$. Let $x \in S(\g)^\h$ be a monomial in the standard PBW ordering; we then consider four cases:

(1) $x \in S(\h)$. In this case, $\pi(x) = x$ and since $D(x) \in S(\h_\epsilon)$, we have $\omega D \pi(x) = \omega D(x) = \pi' \omega D(x)$.

(2) $x$ is of the form $f_\beta h_{\alpha_1}h_{\alpha_2} \dots h_{\alpha_k} e_\beta$, in which case:
\begin{align*}
\omega D(\pi(x)) = &\omega D(0) = 0\\
\pi'(\omega D(x)) = &\pi'(\omega(\overline{f_\beta h_{\alpha_1}h_{\alpha_2} \dots h_{\alpha_k}} e_\beta + \overline{h_{\alpha_1}h_{\alpha_2} \dots h_{\alpha_k} e_\beta} f_\beta\\
&+ \sum_j \overline{f_\beta h_{\alpha_1}h_{\alpha_2} \dots h_{\alpha_{j-1}} h_{\alpha{j+1}} \dots h_{\alpha_k} e_\beta} h_{\alpha_j}))\\
= &\frac{1}{2}\overline{h_{\alpha_1}h_{\alpha_2} \dots h_{\alpha_k} h_\beta} + \frac{1}{2}\overline{h_{\alpha_1}h_{\alpha_2} \dots h_{\alpha_k} h_\beta} + 0 = \overline{h_{\alpha_1}h_{\alpha_2} \dots h_{\alpha_k} h_\beta}.
\end{align*}

To show that second equality, we first observe that $\pi'$ and $\omega$ are linear. Considering the first term, we see that for any $0 \leq i \leq j \leq k$, we have:
\begin{align*}
\pi'(\overline{h_{\alpha_1} \dots h_{\alpha_i}} e_\beta \overline{h_{\alpha_{i+1}} \dots h_{\alpha_j} f_\beta h_{\alpha_{j+1}} \dots h_{\alpha_k}}) = &\pi'(\overline{h_{\alpha_1} \dots h_{\alpha_i}} e_\beta \overline{h_{\alpha_{i+1}} \dots h_{\alpha_k} f_\beta}) \\
= &\pi'(\overline{h_{\alpha_1} \dots h_{\alpha_j} h_{\alpha_{j+1}} \dots h_{\alpha_k}} [e_\alpha, \fbar_\alpha] \\
&+ \sum_{l \geq i} \overline{h_{\alpha_1} \cdots h_{\alpha_l}} [e_\alpha, \overline{h}_{\alpha_{l+1}}] \overline{h_{\alpha_{l+2}} \dots h_{\alpha_k} f_\beta} \\
&+ \overline{f_\beta h_{\alpha_1} h_{\alpha_2} \dots h_{\alpha_k}} e_\beta) \\
= &\pi'(\overline{h_{\alpha_1} h_{\alpha_2} \dots h_{\alpha_k} h_\beta} \\
&+ \sum_{l \geq i} \beta(h_{\alpha_l}) \overline{f_\beta h_{\alpha_1} h_{\alpha_2} \dots h_{\alpha_k} e_\beta} \\
&+ \overline{f_\beta h_{\alpha_1} h_{\alpha_2} \dots h_{\alpha_k}} e_\beta) \\
= &\overline{h_{\alpha_1} h_{\alpha_2} \dots h_{\alpha_k} h_\beta}
\end{align*}
and
\begin{align*}
\pi'(\overline{h_{\alpha_1} \dots h_{\alpha_i} f_\beta h_{\alpha_{i+1}} \dots h_{\alpha_j}} e_\beta \overline{h_{\alpha_{j+1}} \dots h_{\alpha_k}}) = &\pi'(\overline{f_\beta h_{\alpha_1} \dots h_{\alpha_j}} e_\beta \overline{h_{\alpha_{j+1}} \dots h_{\alpha_k}})\\
= &\pi'(\sum_{l \geq j} \overline{f_\beta h_{\alpha_1} \dots h_{\alpha_l}} [e_\beta, \overline{h}_{\alpha_{l+1}}]\overline{h_{\alpha_{l+2}} \dots h_{\alpha_k}} \\
&+ \overline{f_\beta h_{\alpha_1} h_{\alpha_2} \dots h_{\alpha_k}} e_\beta)\\
= &\pi'(\sum_{l \geq j} \beta(h_{\alpha_l}) \overline{f_\beta h_{\alpha_1} h_{\alpha_2} \dots h_{\alpha_k} e_\beta} \\
&+ \overline{f_\beta h_{\alpha_1} h_{\alpha_2} \dots h_{\alpha_k}} e_\beta)\\
= &0.    
\end{align*}
Now, $\omega(\overline{f_\beta h_{\alpha_1}h_{\alpha_2} \dots h_{\alpha_k}} e_\beta)$ is a sum of $(k+2)!$ terms. Half of these terms are of the form 
\[\frac{1}{(k+2)!}\overline{h_{\alpha_1} \dots h_{\alpha_i}} e_\beta \overline{h_{\alpha_{i+1}} \dots h_{\alpha_j} f_\beta h_{\alpha_{j+1}} \dots h_{\alpha_k}}\]
and half are of the form
\[\frac{1}{(k+2)!}\overline{h_{\alpha_1} \dots h_{\alpha_i} f_\beta h_{\alpha_{i+1}} \dots h_{\alpha_j}} e_\beta \overline{h_{\alpha_{j+1}} \dots h_{\alpha_k}}\]
so $\pi'(\omega(\overline{f_\beta h_{\alpha_1}h_{\alpha_2} \dots h_{\alpha_k}} e_\beta)) = \frac{1}{2}\overline{h_{\alpha_1} h_{\alpha_2} \dots h_{\alpha_k} h_\beta}$. Similar results hold for the other terms $\pi'(\omega(\overline{h_{\alpha_1}h_{\alpha_2} \dots h_{\alpha_k} e_\beta} f_\beta))$ and $\pi'(\omega(\overline{f_\beta h_{\alpha_1}h_{\alpha_2} \dots h_{\alpha_{j-1}} h_{\alpha{j+1}} \dots h_{\alpha_k} e_\beta} h_{\alpha_j}))$.

(3) $x$ is of the form $f_{\alpha_1} \dots f_{\alpha_m}h_{\beta_1} \dots h_{\beta_k} e_\gamma$ where $m>1$ and $\gamma = \sum \alpha_i$ (so in particular $\gamma > \alpha_i$ for all $i$). In this case:
\begin{align*}
\omega D(\pi(x)) = &\omega D(0) = 0\\
\pi'(\omega D(x)) = &\pi'(\omega(\sum_i \overline{f_{\alpha_1} \dots f_{\alpha_{i-1}} f_{\alpha_{i+1}} \dots f_{\alpha_m} h_{\beta_1} \dots h_{\beta_k} e_\gamma} f_{\alpha_i} \\
&+ \sum_j \overline{f_{\alpha_1} \dots f_{\alpha_m}h_{\beta_1} \dots h_{\beta_{j-1}} h_{\beta{j+1}} \dots h_{\beta_k} e_\gamma} h_{\beta_j}\\
&+ \overline{f_{\alpha_1} \dots f_{\alpha_m}h_{\beta_1} \dots h_{\beta_k}} e_\gamma))\\
= &0 + 0 + 0
\end{align*}

(4) $x$ is of the form $f_{\alpha_1} \dots f_{\alpha_m}h_{\beta_1} \dots h_{\beta_k} e_{\gamma_1} \dots e_{\gamma_p}$, where $m \geq 1, p > 1$ and $\sum \alpha_i = \sum \gamma_k$. In this case:
\begin{align*}
\omega D(\pi(x)) = &\omega D(0) = 0\\
\pi'(\omega D(x)) = &\pi'(\omega (\sum_i \overline{f_{\alpha_1} \dots f_{\alpha_{i-1}} f_{\alpha_{i+1}} \dots f_{\alpha_m} h_{\beta_1} \dots h_{\beta_k} e_{\gamma_1} \dots e_{\gamma_p}} f_{\alpha_i}\\
&+ \sum_j \overline{f_{\alpha_1} \dots f_{\alpha_m}h_{\beta_1} \dots h_{\beta_{j-1}} h_{\beta{j+1}} \dots h_{\beta_k} e_{\gamma_1} \dots e_{\gamma_p}} h_{\beta_j}\\
&+ \sum_l \overline{f_{\alpha_1} \dots f_{\alpha_m}h_{\beta_1} \dots h_{\beta_k} e_{\gamma_1} \dots e_{\gamma_{l-1}} e_{\gamma_{l+1}} \dots e_{\gamma_p}} e_{\gamma_k}))\\
= &0 + 0 + 0
\end{align*}

In particular, $\omega D\pi(x) - \pi' \omega D(x) \in S(\overline{\h})$ for any monomial $x$ and hence for any $z \in S(\g)^G \subseteq S(\g)^\h$. Hence $\xi_\mu(\lambda) - \psi_\mu(\lambda)$ depends only on $\mu$, so for fixed $\mu$ we have that $\xi_\mu(\lambda) - \psi_\mu(\lambda)$ is a constant as required.

By considering the inclusion map $S(\h)^W \hookrightarrow S(\h)$ and identifying $S(\h)^W$ with $\C[\pi(z_1), \pi(z_2), \dots , \pi(z_n)]$, we can write $\phi: \h^* \rightarrow \h^*/W \cong \C^n$ in components as:
\[\phi(\lambda) = \begin{pmatrix} \pi(z_1)(\lambda) \\
\vdots \\
\pi(z_n)(\lambda)
\end{pmatrix}\]
Now let $\phi_i : \h^* \rightarrow \C$ be given by $\phi_i (\lambda) = \pi(z_i)(\lambda)$. Then we have the following equality, which is the key step in allowing us to understand the action of $Z(\g_\epsilon)$ in terms of the adjoint quotient map $\phi$: 
\[d_\mu \phi_i (\lambda)= \sum_{j = 1}^n \frac{\partial \phi_i}{\partial h_j}(\mu) h_j (\lambda) = D(\pi(z_i))(\lambda, \mu)\]
so $d_\mu \phi (\lambda) = \psi_\mu(\lambda)$ proving part (c).

To show part (d), we invoke a result (see \cite[Proposition 1.2]{R}) that states that if $\xi: \g \rightarrow \C^n$ is the adjoint quotient map, then $\rank(d_\mu \xi') = \dimension(\mathfrak{z}(\g^\mu))$. Now, consider the following diagrams:
\[
\begin{tikzcd}
                                          & S(\h) \arrow[rd, hook] &       &      & \h \arrow[ld, "\phi"'] &                                              \\
S(\h)^W \arrow[rr, hook] \arrow[ru, hook] &                        & S(\g) & \C^n &                        & \g \arrow[lu, "\zeta"'] \arrow[ll, "\xi"]
\end{tikzcd}
\]
where $\zeta: \g = \n^- \oplus \h \oplus \n \rightarrow \h$ is the projection along this decomposition, and we identify $\h$ and $\g$ with their duals $\h^*$ and $\g^*$ respectively using the isomorphism obtained from the Killing form. The diagram on the right commutes since it is the induced by the diagram on the left. In particular, $\rank(d_\mu \phi) = \rank(d_\mu \xi)$, completing the proof of (d).

Finally, to show (e) we let $\Delta_\mu = \{\alpha \in \Phi^+ \cap \Phi_\mu: \alpha \mbox{ is minimal in } \Phi^+ \cap \Phi_\mu\}$, which is a simple system for $\Phi_\mu$, and note that $\dimension(\C\Phi_\mu) = \dimension([\g^\mu, \g^\mu] \cap \h)$. Now consider $ M_{\lambda, \mu}$ and let $\alpha \in \Delta_\mu$. In the case $\mu$ satisfies the hypotheses of Theorem \ref{maintheorem}, $\overline{f_\alpha} \otimes 1_{\lambda, \mu} \in M_{\lambda, \mu}$ is highest weight of weight $(\lambda - \alpha, \mu)$ and generates a submodule of $M_{\lambda, \mu}$ isomorphic to $M_{\lambda - \alpha, \mu}$. Hence by parts (a) and (b), $\alpha \in \ker(\psi_\mu)$ for each $\alpha \in \Delta_\mu$, so by linearity, $\C\Phi_\mu \subseteq \ker(\psi_\mu)$. But by (d), $\dimension(\ker(\psi_\mu)) = \rank(\g) - \dimension(\mathfrak{z}(\g^\mu)) = \dimension([\g^\mu, \g^\mu] \cap \h) = \dimension(\C\Phi_\mu)$ so we must have that $\C\Phi_\mu = \ker(\psi_\mu)$. 
\end{proof}

\begin{proof}[Proof of Theorem \ref{vermaexttheorem}]
Suppose $N_1$ and $N_2$ are highest weight modules of weights $(\lambda, \mu)$ and $(\lambda', \mu')$ respectively. If $\Ext_{\calO_\epsilon}(N_1, N_2) \neq 0$, then by Corollary \ref{calOdecomposition} we must have $\mu = \mu'$. The central characters of $N_1$ and $N_2$ must be the same, so by Lemma \ref{centralcharaterlemma}(a) we must have $\xi_\mu(\lambda) = \xi_\mu(\lambda')$. By Lemma \ref{centralcharaterlemma}(e), we must then have $\lambda - \lambda' \in \C\Phi_\mu$. Observe that for any module $M \in \calO_\epsilon$ and $\nu \in \h^*$, the subspace $\bigoplus_{\nu - \nu' \in \Z\Phi} M^{\nu'}$ is in fact a submodule of $M$, and $M$ is a direct sum of such submodules. Hence $\lambda - \lambda' \in \C\Phi_\mu \cap \Z\Phi = \Z\Phi_\mu$ as required. Now, suppose $M \in \calO_\epsilon^\mu$ is indecomposable and $0 = M_0 \subseteq M_1 \subseteq \dots \subseteq M_n = M$ is a filtration such that each $M_i/M_{i-1}$ is a highest weight module of weight $(\lambda_i, \mu_i)$. Then for each $0 \leq i \leq {n-1}$ there is a short exact sequence $0 \rightarrow M_{i+1}/M_i \rightarrow M_{i+1}/M_{i-1} \rightarrow M_i/M_{i-1} \rightarrow 0$, so the first part of the theorem implies the desired result.
\end{proof}

\subsection{Exactness of the restriction functor}

We now fix a standard parabolic subalgebra $\p$ of $\g$ with Levi decomposition $\p = \mathfrak{l} \oplus \mathfrak{r}$ and let $\mu \in \h^*$ be such that $\g^\mu = \mathfrak{l}$. We prove Theorem \ref{maintheorem} using Theorem \ref{vermaexttheorem} along with the following three lemmas.

\begin{lemma}
\label{highestweightmodinvariants}
Let $M$ be a highest weight module of weight $(\lambda, \mu)$. Then 
\[M^{\mathfrak{r}_\epsilon} = \bigoplus_{\lambda - \lambda' \in \C\Phi_\mu} M^{\lambda'}\]
\end{lemma}

\begin{proof}
We first observe that since $\g^\mu$ is the Levi factor of a standard parabolic $\mathfrak{r} = \vspan\{e_\alpha: \alpha \in \Phi^+ \backslash \Phi_\mu\}$. Now, if $v \in M^{\lambda'}$ for some $\lambda'$ such that $\lambda - \lambda' \in \C\Phi_\mu$, then for $\alpha \in \Phi^+ \backslash \Phi_\mu$ we have $e_\alpha \cdot v, \overline{e_\alpha} \cdot v \in M^{\lambda' + \alpha}$. But since $ \lambda \ngeq \lambda' + \alpha$, $M^{\lambda' + \alpha} = 0$ and so $v \in M^{\mathfrak{r}_\epsilon}$. Hence we have that $M^{\mathfrak{r}_\epsilon} \supseteq \bigoplus_{\lambda - \lambda' \in \C\Phi_\mu} M^{\lambda'}$.

On the other hand, suppose $v \in M^{\mathfrak{r}_\epsilon}$ is of weight $\lambda'$. Then, since $\n_\epsilon \cap \mathfrak{l}_\epsilon$ acts locally finitely and $M^{\mathfrak{r}_\epsilon}$ is a $\g^\mu_\epsilon$ submodule, we can repeatedly apply elements $e_\alpha, \overline{e_\alpha}$ where $\alpha \in \Phi_\mu$ to find a maximal vector whose weight is in $\lambda' + \C\Phi_\mu$. Hence by Corollary \ref{maxhwcorollary} there is a highest weight vector in $M$ whose weight is in $\lambda' + \C\Phi_\mu$, which generates a highest weight submodule of $M$ that must have the same central character as $M$. But by Theorem \ref{vermaexttheorem}, this highest weight module has the same central character only if the weight of $v$ is in $\lambda + \C\Phi_\mu$, which implies $\lambda - \lambda' \in \C\Phi_\mu$ as required.
\end{proof}

\begin{lemma}
\label{invariantsexactsequence}
Let $0 \rightarrow L \overset{f}{\rightarrow} M \overset{g}{\rightarrow} N \rightarrow 0$ be a short exact sequence in $\calO_\epsilon^\mu$. Suppose there exists some non-zero $v \in M^{\mathfrak{r}_\epsilon}$ of weight $\lambda$. There either there exists some non-zero $v_1 \in L^{\mathfrak{r}_\epsilon}$ of weight $\lambda$, or there exists some non-zero $v_2 \in N^{\mathfrak{r}_\epsilon}$ of weight $\lambda$.
\end{lemma}

\begin{proof}
If $v \notin \ker(g)$, then $g(v)$ has weight $\lambda$ and is in $N^{\mathfrak{r}_\epsilon}$. If $v \in \ker(g)$, then since $\ker(g) = \im(f)$, we have $v = f(v_1)$ for some $v_1 \in L$, and since $f$ is injective this $v_1$ must be in $L^{\mathfrak{r}_\epsilon}$ and have weight $\lambda$.
\end{proof}

We can now prove exactness of the functor $R$ defined earlier:

\begin{lemma}
Let $\p$ be a standard parabolic subalgebra of $\g$ with Levi decomposition $\p = \mathfrak{l} \oplus \mathfrak{r}$, and let $\mu \in \h^*$ be such that $\g^\mu = \mathfrak{l}$. Then the functor $(-)^{\mathfrak{r}_\epsilon} = R: \calO_\epsilon^\mu(\g) \rightarrow \calO_\epsilon^\mu(\g^\mu)$ is exact.
\end{lemma}

\begin{proof}
\label{functorexactness}
Let $0 \rightarrow L \overset{f}{\rightarrow} M \overset{g}{\rightarrow} N \rightarrow 0$ be a short exact sequence in $\calO_\epsilon^\mu(\g)$.

Since $f$ is injective, the restriction of $f$ to $L^{\mathfrak{r}_\epsilon}$ is still injective. If $m \in \ker(g) \cap M^{\mathfrak{r}_\epsilon}$, then $m = f(l)$ for some $l \in L$. Additionally, if $r \in \mathfrak{r}_\epsilon$, then $f(r \cdot l) = r \cdot m = 0$, so $r \cdot l \in \ker(f) = 0$. Hence $l \in L^{\mathfrak{r}_\epsilon}$. Therefore $0 \rightarrow L^{\mathfrak{r}_\epsilon} \overset{f}{\rightarrow} M^{\mathfrak{r}_\epsilon} \overset{g}{\rightarrow} N^{\mathfrak{r}_\epsilon}$ is exact, i.e. taking $\mathfrak{r}_\epsilon$ invariants is always left exact.

Hence we only need to show that $g: M^{\mathfrak{r}_\epsilon} \rightarrow N^{\mathfrak{r}_\epsilon}$ is surjective, and by Theorem \ref{vermaexttheorem} it suffices to consider the case where $M$ and $N$ both have filtrations by highest weight modules of weights $(\lambda_i, \mu)$, where $\lambda_i - \lambda_j \in \Z\Phi_\mu$ for all $i, j$. Now, let $v \in N^{\mathfrak{r}_\epsilon}$ have weight $\lambda$. By Lemmas \ref{highestweightmodinvariants} and \ref{invariantsexactsequence}, we must have $\lambda \in \lambda_i + \Z\Phi_\mu$ for some (in fact, any) $\lambda_i$. There exists some $w \in g^{-1}(v)$ which is also of weight $\lambda$, so to complete the proof it is enough to show that any element of $M$ of weight $\lambda \in \lambda_i + \Z\Phi_\mu$ is in $M^{\mathfrak{r}_\epsilon}$. But the weight of any element of $M$ lies in $\bigcup (\lambda_i - \Z_{\geq 0} \Phi^+)$. In particular, if $\alpha \in \Phi^+ \backslash \Phi_\mu$ and $\lambda \in \lambda_i + \C\Phi_\mu$, then $\lambda + \alpha \nleq \lambda_i$ for any $\lambda_i$, so $M^{\lambda + \alpha} = 0$. Hence if $v \in M$ is of weight $\lambda \in \lambda_i + \C\Phi_\mu$, then for any $e_\alpha \in \mathfrak{r}_\epsilon$, $e_\alpha \cdot v \in M^{\lambda + \alpha} = 0$, and similarly $\overline{e_\alpha} \cdot v = 0$, so $v \in M^{\mathfrak{r}_\epsilon}$ as required.
\end{proof}

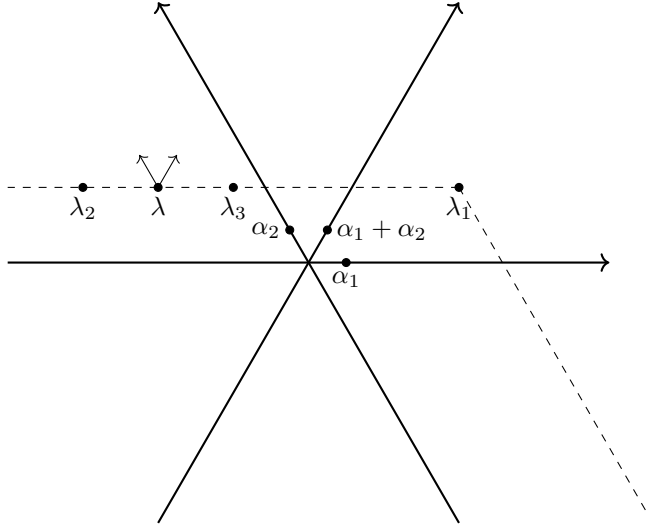
\begin{center}
\begin{tikzpicture}
\draw[thick,->] (-4, 0) -- (4, 0);
\draw[thick,->] (-2, -3.464) -- (2, 3.464);
\draw[thick,->] (2, -3.464) -- (-2, 3.464);
\filldraw[black] (-3, 1) circle (1.5pt) node[anchor = north]{$\lambda_2$};
\filldraw[black] (2, 1) circle (1.5pt) node[anchor = north]{$\lambda_1$};
\filldraw[black] (-2, 1) circle (1.5pt) node[anchor = north]{$\lambda$};
\filldraw[black] (-1, 1) circle (1.5pt) node[anchor = north]{$\lambda_3$};
\filldraw[black] (0.5, 0) circle (1.5pt) node[anchor = north]{$\alpha_1$};
\filldraw[black] (0.25, 0.433) circle (1.5pt) node[anchor = west]{$\alpha_1 + \alpha_2$};
\filldraw[black] (-0.25, 0.433) circle (1.5pt) node[anchor = east]{$\alpha_2$};
\draw[dashed] (2, 1) -- (4.577, -3.464);
\draw[dashed] (2, 1) -- (-4, 1);
\draw[->] (-2, 1) -- (-2.25, 1.433);
\draw[->] (-2, 1) -- (-1.75, 1.433);
\end{tikzpicture}
\end{center}

Above is an illustration of the argument in the case where $\g = \mathfrak{sl}_3$ and $\Phi_\mu = \{\alpha_1, -\alpha_1\}$ for $\alpha_1$ a simple root. Here $M \in \calO_\epsilon^\mu$ has a filtration where the sections are highest weight of weights $(\lambda_1, \mu)$, $(\lambda_2, \mu)$, and $(\lambda_3, \mu)$. If $M^\lambda \neq 0$ for some $\lambda \in \h^*$, then $\lambda$ must lie in the area below the two lines. If $\lambda \in \lambda - \Z_{\geq 0} \Phi_\mu$, then $\lambda + \alpha_2$ and $\lambda + (\alpha_1 + \alpha_2)$ lie above the horizontal line, so for any $m \in M^\lambda$, we have $e_{\alpha_2} \cdot m = \overline{e}_{\alpha_2} \cdot m = e_{\alpha_1 + \alpha_2} \cdot m = \overline{e}_{\alpha_1 + \alpha_2} \cdot m = 0$, i.e $m \in M^{\mathfrak{r}_\epsilon}$.

\begin{proof}[Proof of Theorem \ref{maintheorem}]
As earlier, we let $\phi_M : M \longrightarrow (U(\g_\epsilon) \otimes_{U(\p_\epsilon)} M)^{\mathfrak{r}_\epsilon}$ be given by $\phi_M(m) = 1 \otimes m$ and let $\varphi_N : U(\g_\epsilon) \otimes_{U(\p_\epsilon)} N^{\mathfrak{r}_\epsilon} \longrightarrow N$ be given by $\varphi_N(u \otimes n) = u \cdot n$. We wish to show that these are always isomorphisms. The exactness of the functors $I$ and $R$ combined with Lemma \ref{finitefiltration} and a standard argument using the Five Lemma means it suffices to check this only on highest weight modules.

The map $\phi_M$ is always injective since if $\{m_i : i \in I\}$ is a basis for $M$, then there is a basis for $U(\g_\epsilon) \otimes_{U(\p_\epsilon)} M$ consisting of all elements of the form $u \otimes m_i$ where $u$ is a monomial in $U(\mathfrak{r}_\epsilon^-)$. Hence the element $1 \otimes m \in U(\g_\epsilon) \otimes_{U(\p_\epsilon)} M$ is never 0 if $m$ is non-zero.

Suppose $\phi_M$ is not surjective. Then there exist $u_1, \dots, u_n \in U(\mathfrak{r}_\epsilon^-)$ and $m_1, \dots m_n \in M$ such that not all the $u_k$ are scalars, the $m_k$ are weight vectors, and $\sum u_k \otimes m_k \in (U(\g_\epsilon) \otimes_{U(\p_\epsilon)} M)^{\mathfrak{r}_\epsilon}$. But if $M$ is highest weight of weight $(\lambda, \mu)$ with highest weight generator $m$, then $U(\g_\epsilon) \otimes_{U(\p_\epsilon)} M$ is highest weight of the same weight since $1 \otimes m$ will be a highest weight generator. By Lemma \ref{highestweightmodinvariants}, we have:
\[(U(\g_\epsilon) \otimes_{U(\p_\epsilon)} M)^{\mathfrak{r}_\epsilon} = \bigoplus_{\lambda - \lambda' \in \C\Phi_\mu} (U(\g_\epsilon) \otimes_{U(\p_\epsilon)}M)^{\lambda'}\]
but if $u_k$ is not a scalar, the weight $\lambda'$ of $u_k \otimes m_k$ does not satisfy $\lambda - \lambda' \in \C \Phi_\mu$, giving a contradiction. Hence $\phi_M$ is an isomorphism for any $M \in \calO_\epsilon^\mu(\g)$.

We observe that if $n \in N$ is a highest weight generator of $N$ then certainly $n \in N^{\mathfrak{r}_\epsilon}$, so $n = \varphi_N(1 \otimes n) \in \im(\varphi_N)$ and hence $\varphi_N$ is surjective. To see it is injective, let $K = \ker(\varphi_N)$ and consider the short exact sequence:
\[0 \rightarrow K \hookrightarrow U(\g_\epsilon) \otimes_{U(\p_\epsilon)} N^{\mathfrak{r}_\epsilon} \overset{\varphi_N}{\rightarrow} N \rightarrow 0.\]
Since the functor $R$ is exact, we have another exact sequence:
\[0 \rightarrow K^{\mathfrak{r}_\epsilon} \hookrightarrow (U(\g_\epsilon) \otimes_{U(\p_\epsilon)} N^{\mathfrak{r}_\epsilon})^{\mathfrak{r}_\epsilon} \rightarrow N^{\mathfrak{r}_\epsilon} \rightarrow 0\]
where the map $(U(\g_\epsilon) \otimes_{U(\p_\epsilon)} N^{\mathfrak{r}_\epsilon})^{\mathfrak{r}_\epsilon} \rightarrow N^{\mathfrak{r}_\epsilon}$ is the restriction of $\varphi_N$ to $(U(\g_\epsilon) \otimes_{U(\p_\epsilon)} N^{\mathfrak{r}_\epsilon})^{\mathfrak{r}_\epsilon}$. Now, we observe that $\varphi_N(\phi_{N^{\mathfrak{r}_\epsilon}}(n)) = \varphi_N(1 \otimes n) = n$, so $\varphi_N$ restricted to $(U(\g_\epsilon) \otimes_{U(\p_\epsilon)} N^{\mathfrak{r}_\epsilon})^{\mathfrak{r}_\epsilon}$ is the inverse of $\phi_{N^{\mathfrak{r}_\epsilon}}$ and hence is bijective. Therefore $K^{\mathfrak{r}_\epsilon} = 0$, and so $K = 0$ since if $K$ were non-zero, it would contain a non-zero highest weight vector $v$ which would then be in $K^{\mathfrak{r}_\epsilon}$.
\end{proof}

\section{Twisting Functors}

\subsection{Definition of twisting functors}

Now we aim to prove Theorem \ref{twistingtheoremshort} using a generalisation of the twisting functors for BGG category $\calO$. Fix some $\alpha \in \Delta$ and let $U = U(\g_\epsilon)$, the enveloping algebra of $\g_\epsilon$. We let
\[v(\mathbf{k}, \mathbf{l}, \mathbf{m}, \mathbf{n}) = \prod_{\beta \in I} e_\beta^{k_\beta} \prod_{\beta \in I} \overline{e}_\beta^{l_\beta} \prod_{\beta \in J} h_\beta^{m_\beta} \prod_{\beta \in J} \overline{h}_\beta^{n_\beta} \in U\]
for any $\mathbf{k}, \mathbf{l} \in \Z_{\geq 0}^I$ and $\mathbf{m}, \mathbf{n} \in \Z_{\geq 0 }^J$, where $I = \Phi \backslash \{\alpha\}$ and $J = \Delta$.

We need to consider the localisation of $U$ with respect to the subset $F_\alpha = \{f_\alpha^i \fbar_\alpha^j : i, j \geq 0\}$. Since the ring $U$ has no zero divisors, by \cite[Theorem 2.1.12]{MR} to show the left localisation exists it suffices to show that $F_\alpha$ is a left Ore set, i.e. that for all $u \in U$ and $f \in F_\alpha$, $F_\alpha u \cap U f \neq \emptyset$. In fact, we will also show that $F_\alpha$ is a right Ore set, so
the right localisation also exists and is isomorphic to the left localisation.

\begin{lemma}
Let $\alpha \in \Delta$. The sets $\{f_\alpha^i: i \geq 0\}$, $\{\fbar_\alpha^j: j \geq 0\}$ and $F_\alpha = \{f_\alpha^i \fbar_\alpha^j : i, j \geq 0\}$ are all both left and right Ore sets in $U$.
\end{lemma}

\begin{proof}
We show that each of these sets is a left Ore set; the proofs that they are also right Ore sets are very similar. Let $x \in \g_\epsilon \subseteq U$ and let $j \geq 0$. It is easily verified by induction that:
\[f_\alpha^j x = \sum_{k=0}^j \binom{j}{k} (\ad{f_\alpha})^k(x) f_\alpha^{j-k}\]
But $\ad(f_\alpha)$ is nilpotent, so there exists $l$ such that $(\ad{f_\alpha})^l = 0$. Hence for sufficiently large $j$, we have $f_\alpha^j x = u_j f_\alpha^{j-l}$ for some $u_j \in U$. Applying this repeatedly, we see that given a monomial $v \in U$ and some $i \geq 0$, we can find $u \in U$ and $j \geq 0$ such that $f_\alpha^j v = u f_\alpha^i$. But $U$ is spanned by such monomials, so we have shown that $\{f_\alpha^i : i \geq 0\}$ is a left Ore set. The proof that $\{\fbar_\alpha^j : j \geq 0\}$ is a left Ore set is identical, since $\ad(\fbar_\alpha)$ is also nilpotent.

Finally, to see that $\{f_\alpha^i \fbar_\alpha^j : i, j \geq 0\}$ is a left Ore set, we let $i, j \geq 0$ and let $u \in U$. We must find $u' \in U$, $i', j' \geq 0$ such that $f_\alpha^{i'} \fbar_\alpha^{j'} u = u' f_\alpha^i \fbar_\alpha^j$. Since $\{\fbar_\alpha^j : j \geq 0\}$ is a left Ore set, there exist $j' \geq 0$ and $u'' \in U$ such that $u'' \fbar_\alpha^j = \fbar_\alpha^{j'} u$, and since $\{f_\alpha^i : i \geq 0\}$ is a left Ore set there exist $i' \geq 0$ and $u' \in U$ such that $u' f_\alpha^{i} = f_\alpha^{i'} u''$. Hence we have $u' f_\alpha^i \fbar_\alpha^j = f_\alpha^{i'} u'' \fbar_\alpha^j = f_\alpha^{i'} \fbar_\alpha^{j'} u$ as required.
\end{proof}

We now define $U_\alpha$ to be the localisation of $U$ with respect to $F_\alpha$. The map $\iota : U \rightarrow U_\alpha$ is injective because $U$ has no zero divisors, so $U_\alpha$ has the structure of a $U$-$U$-bimodule. We wish to consider two possible choices of basis for $U_\alpha$, given by $\{f_\alpha^i \fbar_\alpha^j v(\mathbf{k}, \mathbf{l}, \mathbf{m}, \mathbf{n}) : i,j \in \Z, \mathbf{k}, \mathbf{l} \in \Z_{\geq 0}^I, \mathbf{m}, \mathbf{n} \in \Z_{\geq 0}^J\}$ and $\{v(\mathbf{k}, \mathbf{l}, \mathbf{m}, \mathbf{n}) f_\alpha^i \fbar_\alpha^j: i,j \in \Z, \mathbf{k}, \mathbf{l} \in \Z_{\geq 0}^I, \mathbf{m}, \mathbf{n} \in \Z_{\geq 0}^J\}$. These sets can be seen to be bases by using the PBW theorem and the fact that $F_\alpha$ is both a left Ore set and a right Ore set.

\begin{lemma}
\label{relationslemma}
For any $\beta \in \Phi^+$, the following relations hold in $U_\alpha$.
\begin{align}
[e_\alpha, f_\alpha^{-1}] &= f_\alpha^{-2}(-h_\alpha - 2) \\
[e_\alpha, \fbar_\alpha^{-1}] &= - \fbar_\alpha^{-2} \overline{h}_\alpha \\
[\overline{e}_\alpha, f_\alpha^{-1}] &= - f_\alpha^{-2} \overline{h}_\alpha - 2f_\alpha^{-3}\fbar_\alpha \\
[\overline{e}_\alpha, \fbar_\alpha^{-1}] &= 0
\end{align}
For any $h \in \h$:
\begin{align}
[h, f_\alpha^{-1}] &= \alpha(h) f_\alpha^{-1} \\
[h, \fbar_\alpha^{-1}] &= \alpha(h) \fbar_\alpha^{-1} \\
[\overline{h}, f_\alpha^{-1}] &= \alpha(h) f_\alpha^{-2} \fbar_\alpha\\
[\overline{h}, \fbar_\alpha^{-1}] &= 0
\end{align}
For any $\beta \in \Phi^+ \backslash \{\alpha\}$:
\begin{align}
[e_\beta, f_\alpha^{-1}] &= l_1 f_\alpha^{-2} e_{\beta - \alpha} + l_2 f_\alpha^{-3} e_{\beta - 2\alpha} + l_3 f_\alpha^{-4} e_{\beta - 3\alpha} \\
[e_\beta, \fbar_\alpha^{-1}] &= m_1 \fbar_\alpha^{-2} \overline{e}_{\beta - \alpha} \\
[\overline{e}_\beta, f_\alpha^{-1}] &= n_1  f_\alpha^{-2} \overline{e}_{\beta - \alpha} + n_2 f_\alpha^{-3} \overline{e}_{\beta - 2\alpha} + n_3 f_\alpha^{-4} \overline{e}_{\beta - 3\alpha} \\
[\overline{e}_\beta, \fbar_\alpha^{-1}] &= 0
\end{align}
for some $l_i, m_i, n_i \in \C$, letting $e_{\beta'} = 0$ if $\beta' \notin \Phi$.
\end{lemma}
\begin{proof}
These can be verified by multiplying by powers of $f_\alpha$ and $\fbar_\alpha$ to obtain an expression which holds in $U$. We verify relations (5.1), (5.7) and (5.9); the other relations are similar. To see (1), observe that in $U$:
\begin{align*}
f_\alpha^2 e_\alpha &= f_\alpha e_\alpha f_\alpha + f_\alpha [f_\alpha, e_\alpha] \\
&= f_\alpha e_\alpha f_\alpha + f_\alpha (-h_\alpha) \\
&= f_\alpha e_\alpha f_\alpha + (- h_\alpha - 2) f_\alpha.
\end{align*}
We then multiply on the left by $f_\alpha^{-2}$ and on the right by $f_\alpha^{-1}$ to obtain (5.1). To see (5.7), observe that:
\begin{align*}
f_\alpha^2 \overline{h} &= f_\alpha \overline{h} f_\alpha + f_\alpha [f_\alpha, \overline{h}] \\
&= f_\alpha \overline{h} f_\alpha + \alpha(h) \fbar_\alpha f_\alpha,
\end{align*}
and again we obtain the desired relation by multiplying on the left by $f_\alpha^{-2}$ and on the right by $f_\alpha^{-1}$. Finally, to see (5.9), we use that $\beta - 4\alpha \notin \Phi$ and that if $\beta - n \alpha \in \Phi$ for some $n \in \mathbb{N}$, then $\beta - n \alpha \in \Phi^+ \backslash \{\alpha\}$. We then have:
\begin{align*}
f_\alpha^4 e_\beta &= f_\alpha^3 e_\beta f_\alpha + f_\alpha^3 [f_\alpha, e_\beta] \\
&= f_\alpha^3 e_\beta f_\alpha + l_1 f_\alpha^3 e_{\beta - \alpha} \\
&= f_\alpha^3 e_\beta f_\alpha + l_1 f_\alpha^2 e_{\beta - \alpha} f_\alpha + l_1 f_\alpha^2 [f_\alpha, e_{\beta - \alpha}] \\
&= f_\alpha^3 e_\beta f_\alpha + l_1 f_\alpha^2 e_{\beta - \alpha} f_\alpha + l_2 f_\alpha^2 e_{\beta - 2\alpha} \\
&= f_\alpha^3 e_\beta f_\alpha + l_1 f_\alpha^2 e_{\beta - \alpha} f_\alpha + l_2 (f_\alpha e_{\beta - 2\alpha} f_\alpha + f_\alpha[f_\alpha, e_{\beta - 2\alpha}]) \\
&= f_\alpha^3 e_\beta f_\alpha + l_1 f_\alpha^2 e_{\beta - \alpha} f_\alpha + l_2 f_\alpha e_{\beta - 2\alpha} f_\alpha + l_3 e_{\beta - 3\alpha} f_\alpha,
\end{align*}
and multiplying on the left by $f_\alpha^{-4}$ and on the right by $f_\alpha^{-1}$ gives (5.9).
\end{proof}

\begin{lemma}
\label{subbimodulelemma}
Let $U' = \vspan\{f_\alpha^i \fbar_\alpha^j v(\mathbf{k}, \mathbf{l}, \mathbf{m}, \mathbf{n}) : \mbox{ either } i \geq 0 \mbox{ or } j \geq 0, \mathbf{k}, \mathbf{l} \in \Z_{\geq 0}^I, \mathbf{m}, \mathbf{n} \in \Z_{\geq 0}^J\} \subseteq U_\alpha$. Then $U'$ is a $U$-$U$-subbimodule of $U_\alpha$.
\end{lemma}

\begin{proof}
We first show that $U'$ is a right $U$-module. If either $i \geq 0$ or $j \geq 0$, both $k \geq 0$ and $l \geq 0$, and $v = v(\mathbf{k}, \mathbf{l}, \mathbf{m}, \mathbf{n}), v' = v(\mathbf{k}', \mathbf{l}', \mathbf{m}', \mathbf{n}')$ for some $\mathbf{k, l, k', l'} \in \Z_{\geq 0}^I$, $\mathbf{m, n, m', n'} \in \Z_{\geq 0}^J$, then by the PBW theorem we have $v f_\alpha^k \fbar_\alpha^l v = \sum_n (f_\alpha^{k_n} \fbar_\alpha^{l_n} v_n)$ for some $k_n, l_n \geq 0$, and $v_n = v(\mathbf{k}_n, \mathbf{l}_n, \mathbf{m}_n, \mathbf{n}_n)$ for some $\mathbf{k}_n, \mathbf{l}_n \in \Z_{\geq 0}^I$, $\mathbf{m}_n, \mathbf{n}_n \in \Z_{\geq 0}^J$. Hence:
\[f_\alpha^i \fbar_\alpha^j v \cdot (f_\alpha^k \fbar_\alpha^l v') = \sum_n (f_\alpha^i \fbar_\alpha^j f_\alpha^{k_n} \fbar_\alpha^{l_n} v_n) = \sum_n (f_\alpha^{i + k_n} \fbar_\alpha^{j +l_n} v_n)\]
In particular this expression is in $U'$. But $U$ is spanned by elements of the form $f_\alpha^k \fbar_\alpha^l v'$, so $U'$ is a right $U$-module.

We now claim that $U' = \vspan\{v(\mathbf{k}, \mathbf{l}, \mathbf{m}, \mathbf{n}) f_\alpha^i \fbar_\alpha^j : \mbox{ either } i \geq 0 \mbox{ or } j \geq 0, \mathbf{k}, \mathbf{l} \in \Z_{\geq 0}^I, \mathbf{m}, \mathbf{n} \in \Z_{\geq 0}^J\}$; once we have this, it follows that $U'$ is a left $U$-module by a similar argument to above. To see this claim, suppose $i \geq 0$ (the case $j \geq 0$ is very similar) and $v = v(\mathbf{k}, \mathbf{l}, \mathbf{m}, \mathbf{n})$ for some $\mathbf{k}, \mathbf{l} \in \Z_{\geq 0}^I$, $\mathbf{m}, \mathbf{n} \in \Z_{\geq 0}^J$. Using the PBW theorem, we have $v f_\alpha^i =  \sum_n(f_\alpha^{i_n} \fbar_\alpha^{j_n} v_n)$ where $i_n, j_n \geq 0$ and $v_n = v(\mathbf{k}_n \mathbf{l}_n, \mathbf{m}_n, \mathbf{n}_n)$ for some $\mathbf{k}_n, \mathbf{l}_n \in \Z_{\geq 0}^I$ and $\mathbf{m}_n, \mathbf{n}_n \in \Z_{\geq 0}^J$. Using the fact that $\{\fbar_\alpha^{j} : j \geq 0\}$ is an Ore set and the PBW theorem, we also have that $v_n \fbar_\alpha^{j} = \sum_m(\fbar_\alpha^{j_{n, m}} f_\alpha^{i_{n, m}} v_{n, m})$ where $i_{n, m} \in \Z_{\geq 0}$, $j_{n, m} \in \Z$, and $v_{n, m} = v(\mathbf{k}_{n, m}, \mathbf{l}_{n, m}, \mathbf{m}_{m, n}, \mathbf{n}_{n, m})$ for some $\mathbf{k}_{n, m}, \mathbf{l}_{n, m} \in \Z_{\geq 0}^I$, $\mathbf{m}_{m, n}, \mathbf{n}_{n, m} \in \Z_{\geq 0}^J$. Hence we have:
\begin{align*}
v f_\alpha^i \fbar_\alpha^j &= \sum_n(f_\alpha^{i_n} \fbar_\alpha^{j_n} v_n \fbar_\alpha^{j}) \\
&= \sum_n(f_\alpha^{i_n} \fbar_\alpha^{j_n} \sum_m(\fbar_\alpha^{j_{n, m}} f_\alpha^{i_{n, m}} v_{n, m})) \\
&= \sum_{n, m} (f_\alpha^{i_n + i_{n, m}} \fbar_\alpha^{j_n + j_{n, m}} v_{n, m}).
\end{align*}
so $v f_\alpha^i \fbar_\alpha^j \in U'$. Hence
\[U' \supseteq \vspan\{v(\mathbf{k}, \mathbf{l}, \mathbf{m}, \mathbf{n}) f_\alpha^i \fbar_\alpha^j : \mbox{ either } i \geq 0 \mbox{ or } j \geq 0, \mathbf{k}, \mathbf{l} \in \Z_{\geq 0}^I, \mathbf{m}, \mathbf{n} \in \Z_{\geq 0}^J\}.\]
The inclusion $U' \subseteq \vspan\{v(\mathbf{k}, \mathbf{l}, \mathbf{m}, \mathbf{n}) f_\alpha^i \fbar_\alpha^j : \mbox{ either } i \geq 0 \mbox{ or } j \geq 0, \mathbf{k}, \mathbf{l} \in \Z_{\geq 0}^I, \mathbf{m}, \mathbf{n} \in \Z_{\geq 0}^J\}$ follows by a similar argument.
\end{proof}

We now define $S_\alpha:= U_\alpha / U'$, a $U$-$U$-bimodule. Observe that $S_\alpha$ has a basis given by:
\[\{f_\alpha^i \fbar_\alpha^j v(\mathbf{k}, \mathbf{l}, \mathbf{m}, \mathbf{n}) + U': i, j < 0, \mathbf{k}, \mathbf{l} \in \Z_{\geq 0}^I, \mathbf{m}, \mathbf{n} \in \Z_{\geq 0}^J\}\]
and another basis given by:
\[\{v(\mathbf{k}, \mathbf{l}, \mathbf{m}, \mathbf{n}) f_\alpha^i \fbar_\alpha^j + U': i, j < 0, \mathbf{k}, \mathbf{l} \in \Z_{\geq 0}^I, \mathbf{m}, \mathbf{n} \in \Z_{\geq 0}^J\}\]
using the two bases of $U'$ considered in the proof of Lemma \ref{subbimodulelemma}. From now on we omit the +$U'$ when writing elements of $S_\alpha$. At several points we will use the fact that in $S_\alpha$ the elements $f_\alpha^{-i} \fbar_\alpha^{0}$ and $f_\alpha^{0} \fbar_\alpha^{-j}$ are zero.  

Consider the element $s_\alpha \in W = N_G(H)/H$. We lift $s_\alpha$ to an element of $N_G(H)$, which defines an automorphism $\phi_\alpha : \g \rightarrow \g$. The automorphism $\phi_\alpha$ maps $h$ to $s_\alpha(h)$ for any $h \in \h$, and hence must also map $\g_{\beta}$ to $\g_{s_\alpha(\beta)}$ for any $\beta \in \Phi$. Now, let $\phi_\alpha(e_\alpha) = k_1 f_\alpha$ and let $\phi_\alpha(f_\alpha) = k_2 e_\alpha$. Then $- k_1 k_2 h_\alpha = [\phi(e_\alpha), \phi(f_\alpha)] = \phi_\alpha([e_\alpha, f_\alpha]) = \phi_\alpha(h_\alpha) = - h_\alpha$, so $k_1 k_2 = 1$. Hence, rescaling $e_\alpha$ and $f_\alpha$ if necessary, we may assume that:
\begin{align}
\phi_\alpha(e_\alpha) &= f_\alpha\\
\phi_\alpha(f_\alpha) &= e_\alpha.
\end{align}
We can then extend $\phi_\alpha$ to an automorphism $\g_\epsilon \rightarrow \g_\epsilon$ by setting $\phi_\alpha(\overline{x}) = \overline{\phi_\alpha(x)}$ for all $x \in \g$. If $M$ is a left $U$-module, we denote by $\phi_\alpha(M)$ the module obtained by twisting the left action on $M$ by $\phi_\alpha$ and write $\cdot_\alpha$ for this action, so that if $m \in M$ and $u \in U$ then $u \cdot_\alpha m = \phi_\alpha(u) \cdot m$. Similarly, we can twist the left action on $M$ by $\phi_\alpha^{-1}$, and we write $\cdot_{\alpha^{-1}}$ for this action.

Let $\mathcal{C}$ be the full subcategory of $U$-mod whose objects are the $\h$-semisimple modules. We define a functor $\mathcal{H}$ from $U$-mod to $\mathcal{C}$ by letting $\mathcal{H}(M)$ be the maximal $\h$-semisimple submodule of $M$. Using $S_\alpha$ and $\phi_\alpha$, we can then define two functors $T_\alpha, G_\alpha: \mathcal{C} \rightarrow \mathcal{C}$ by:
\begin{align*}
T_\alpha M &= \phi_\alpha(S_\alpha \otimes_U M) \\
G_\alpha M &= \mathcal{H}(\Hom_U(S_\alpha, \phi_\alpha^{-1}(M)))
\end{align*}
where we note that the action of $U$ on $G_\alpha M$ is given by $(u \cdot f)(s) = f(s \cdot u)$ for any $u \in U$, $f \in G_\alpha M$ and $s \in S_\alpha$. We also note that if $M, N \in \mathcal{C}$ and $\chi: M \rightarrow N$, then $T_\alpha(\chi): T_\alpha M \rightarrow T_\alpha N$ is given by $T_\alpha(\chi)(s \otimes m) = s \otimes \chi(m)$ and $G_\alpha(\chi): G_\alpha M \rightarrow G_\alpha N$ is given by $G_\alpha(\chi)(\rho) = \chi \circ \rho$.

In order to see these functors are well defined, the only thing to check is that $T_\alpha M \in \mathcal{C}$ for any $M \in \mathcal{C}$. Let $M \in \mathcal{C}$ and let $\{v_i: i \in I\}$ be a set of weight vectors spanning $M$. Then $T_\alpha M$ is spanned by $\{f_\alpha^{-n} \fbar_\alpha^{-m} \otimes v_i : n,m > 0, i \in I\}$. But if $w \in M^{\lambda}$, then for any $h \in \h$:
\begin{align*}
h \cdot_\alpha (f_\alpha^{-n} \fbar_\alpha^{-m} \otimes w) &= (s_\alpha(h)f_\alpha^{-n} \fbar_\alpha^{-m}) \otimes w \\
&= f_\alpha^{-n} \fbar_\alpha^{-m} (s_\alpha(h) - (n+m) \alpha(h))\otimes w \\
&= f_\alpha^{-n} \fbar_\alpha^{-m} \otimes (s_\alpha(h) - (n+m) \alpha(h)) w \\
&= f_\alpha^{-n} \fbar_\alpha^{-m} \otimes (s_\alpha(\lambda) - (n+m)\alpha)(h) w \\
&= (s_\alpha(\lambda) - (n+m)\alpha)(h) (f_\alpha^{-n} \fbar_\alpha^{-m} \otimes w)
\end{align*}
so for any $w$ of weight $\lambda$ we have
\begin{align}
f_\alpha^{-n} \fbar_\alpha^{-m} \otimes w \in (T_\alpha M)^{s_\alpha(\lambda) - (n+m)\alpha}.
\end{align}
In particular, $T_\alpha M$ is spanned by weight vectors, and so is $\h$-semisimple.

\begin{lemma}
\label{weightfunctionslemma}
Let $g \in \Hom_U(S_\alpha, \phi_\alpha^{-1}(M))$. Then $g$ has weight $\lambda$ if and only if for all $i, j \geq 0$, $g(f_\alpha^{-i} \fbar_\alpha^{-j})$ has weight $\lambda + (i+j)\alpha$ in $\phi^{-1}_\alpha(M)$, or equivalently has weight $s_\alpha(\lambda) - (i+j) \alpha$ in $M$.
\end{lemma}

\begin{proof}
First we assume that $g$ has weight $\lambda$. Then for any $h \in \h$, $i, j \geq 0$:
\begin{align*}
\lambda(h) g(f_\alpha^{-i} \fbar_\alpha^{-j}) &= (h \cdot g)(f_\alpha^{-i} \fbar_\alpha^{-j}) \\
&= g(f_\alpha^{-i} \fbar_\alpha^{-j} h) \\
&= g((h - (i+j)\alpha(h)) f_\alpha^{-i} \fbar_\alpha^{-j}) \\
&= (h - (i+j)\alpha(h)) \cdot g(f_\alpha^{-i} \fbar_\alpha^{-j})
\end{align*}
so $g(f_\alpha^{-i} \fbar_\alpha^{-j})$ has weight $\lambda + (i+j)\alpha$ in $\phi_\alpha^{-1}(M)$, and hence has weight $s_\alpha(\lambda) - (i+j)\alpha$ in $M$.

On the other hand, suppose $g(f_\alpha^{-i} \fbar_\alpha^{-j})$ has weight $s_\alpha(\lambda) - (i+j)\alpha$ in $M$ for all $i, j \geq 0$. Then for any $h \in \h, i, j \geq 0$, $u = v(\mathbf{k}, \mathbf{l}, \mathbf{m}, \mathbf{n})$ for some $\mathbf{k}, \mathbf{l} \in \Z_{\geq 0}^I$, $\mathbf{m}, \mathbf{n} \in \Z_{\geq 0}^J$:
\begin{align*}
(h \cdot g)(u f_\alpha^{-i} \fbar_\alpha^{-j}) &= g(u f_\alpha^{-i} \fbar_\alpha^{-j} h)\\
&=  u \cdot g(f_\alpha^{-i} \fbar_\alpha^{-j} h) \\
&= u \cdot g((h - (i+j)\alpha(h)) f_\alpha^{-i} \fbar_\alpha^{-j}) \\
&= u \cdot (h - (i+j)\alpha(h)) \cdot g(f_\alpha^{-i} f_\alpha^{-j}) \\
&= u \cdot \lambda(h) g(f_\alpha^{-i} \fbar_\alpha^{-j}) \\
&= \lambda(h) g(u f_\alpha^{-i} \fbar_\alpha^{-j})
\end{align*}
and hence $g$ has weight $\lambda$.
\end{proof}

\begin{corollary}
\label{weightfunctionscor}
Let $g \in \Hom_U(S_\alpha, \phi_\alpha^{-1}(M))$, Then $g$ is a weight vector if and only if $g(f_\alpha^{-i} \fbar_\alpha^{-j})$ is a weight vector for all $i, j \geq 0$.  
\end{corollary}

\begin{proof}
Suppose $g \in \Hom_U(S_\alpha, \phi_\alpha^{-1}(M))$ is such that $g(f_\alpha^{-i} \fbar_\alpha^{-j})$ is a weight vector for all $i, j \geq 0$. Since $e_\alpha \cdot_{\alpha^{-1}} g(f_\alpha^{-i} \fbar_\alpha^{-j}) = g(f_\alpha^{-i+1} \fbar_\alpha^{-j})$, in $\phi_\alpha^{-1}(M)$ we have that $\wt(g(f_\alpha^{-i} \fbar_\alpha^{-j})) + \alpha = \wt(g(f_\alpha^{-i+1} \fbar_\alpha^{-j}))$ and similarly we have $\wt(g(f_\alpha^{-i} \fbar_\alpha^{-j})) + \alpha = \wt(g(f_\alpha^{-i} \fbar_\alpha^{-j+1}))$. Hence if in $\phi_\alpha^{-1}(M)$ the weight of $g(f_\alpha^{-1} \fbar_\alpha^{-1})$ is $\lambda'$, we have that the weight of $g(f_\alpha^{-i} \fbar_\alpha^{-j})$ is $\lambda' - 2\alpha + (i+j)\alpha$. Applying Lemma \ref{weightfunctionslemma}, we see that $g$ is indeed a weight vector.
\end{proof}

\begin{lemma}
$T_\alpha$ is right exact and $G_\alpha$ is left exact.
\end{lemma}

\begin{proof}
The fact $T_\alpha$ is right exact follows immediately from the definition, because for any module $M$ we have that $\phi_\alpha(S_\alpha \otimes_U M) \cong \phi_\alpha(S_\alpha) \otimes_U M$ and taking the tensor product with a fixed module always defines a right exact functor. Let $G_\alpha' M := \Hom_U(S_\alpha, \phi_\alpha^{-1}(M))$. Then $G_\alpha'$ and $\mathcal{H}$ are both left exact, so, since composition of left exact functors is left exact, $G_\alpha$ is also left exact.
\end{proof}

\subsection{Equivalence}

Retaining the notation of the previous section, we can now precisely state and prove Theorem \ref{twistingtheoremshort}:
\begin{theorem}
\label{maintwistingtheorem}
Let $\mu \in \h^*$ be such that $\mu(h_\alpha) \neq 0$. Then $T_\alpha$ and $G_\alpha$ give an equivalence of categories between $\calO^\mu$ and $\calO^{s_\alpha(\mu)}$.
\end{theorem}

We prove the following lemma, which immediately implies Theorem \ref{maintwistingtheorem}.
\begin{lemma}
\label{maintwistinglemma}
Let $\mu \in \h^*$ be such that $\mu(h_\alpha) \neq 0$. Then:

\begin{enumerate}
\item[(a)] $T_\alpha$ restricts to a functor $\calO_\epsilon^\mu \rightarrow \calO_\epsilon^{s_\alpha(\mu)}$.

\item[(b)] For any $M \in \calO_\epsilon^\mu$, the map $\psi_M : M \rightarrow G_\alpha T_\alpha M$ given by $\psi_M(m)(s) = s \otimes m$ is an isomorphism

\item[(c)] For any $N \in \calO_\epsilon^{s_\alpha(\mu)}$, the map $\epsilon_N : T_\alpha G_\alpha N \rightarrow N$ given by $\epsilon_N(s \otimes g) = g(s)$ is an isomorphism.

\item[(d)] $G_\alpha$ restricts to a functor $\calO_\epsilon^{s_\alpha(\mu)} \rightarrow \calO_\epsilon^\mu$.

\item[(e)] The transformations $\psi: \id_{\calO_\epsilon^\mu} \rightarrow G_\alpha T_\alpha$ and $\epsilon: T_\alpha G_\alpha \rightarrow \id_{\calO_\epsilon^{s_\alpha(\mu)}}$ are natural.
\end{enumerate}
\end{lemma}

The proof of this lemma requires a series of other results. First, we record the following lemma, which will be useful in the proof of parts (a) and (d):

\begin{lemma}
\label{extensionlemma}
Let $M_1, M_2 \in \calO_\epsilon^\mu$ for some $\mu \in \h^*$, and let $0 \rightarrow M_1 \rightarrow M \rightarrow M_2 \rightarrow 0$ be a short exact sequence of $U(\g_\epsilon)$-modules. Then $M \in \calO_\epsilon^\mu$ if and only if $M$ is $\h$-semisimple.
\end{lemma}

\begin{proof}
By the definition of $\calO_\epsilon^\mu$, if $M \in \calO^\mu$ then $M$ must be $\h$-semisimple. On the other hand, suppose $M$ is $\h$-semisimple. Then to show $M \in \calO_\epsilon^\mu$, we must show:

(i) $M$ is finitely generated.

(ii) For all $h \in \h$, $\overline{h} - \mu(h)$ acts locally nilpotently on $M$.

(iii) The subalgebra $\n_\epsilon \subseteq \g_\epsilon$ acts locally nilpotently on $M$. 

For (i), let $X_1$ be a finite generating set for $M_1$ and let $X_2$ be a finite generating set for $M_2$. Let $X_2'$ be a set containing a choice of preimage for each element of $X_2$. Then $X_1 \cup X_2'$ is a finite generating set for $M$. For (ii), let $m \in M$. Then for some $i \geq 0$, $(\overline{h} - \mu(h))^i \cdot (m + M_1) = 0 \in M/M_1 \cong M_2$, so $(\overline{h} - \mu(h))^i \cdot m = m' \in M_1$, and for some $j \geq 0$, $(\overline{h} - \mu(h))^j \cdot m' = 0$, so we have $(\overline{h} - \mu(h))^{i+j} \cdot m = 0$. Finally, we have that $\dim(M^\lambda) = \dim(M_1^\lambda) + \dim(M_2^\lambda)$ for any $\lambda \in \h^*$, so in particular the weight spaces of $M$ are bounded above. Hence (iii) holds.
\end{proof}

\begin{proof}[Proof of \ref{maintwistinglemma}(a)]

Let $M \in \calO_\epsilon^\mu$, and let $0 = M_0 \subseteq M_1 \subseteq \dots \subseteq M_{k-1} \subseteq M_k = M$ be a filtration of $M$ such that each section is a highest weight module and this filtration has minimal possible length. Such a filtration exists by Lemma \ref{finitefiltration}. We have an exact sequence:
\[0 \rightarrow M_1 \rightarrow M \rightarrow M/M_1 \rightarrow 0\]
and since $T_\alpha$ is right exact, we have that
\[T_\alpha M_1 \rightarrow T_\alpha M \rightarrow T_\alpha(M/M_1) \rightarrow 0\]
is exact. In particular, $T_\alpha M$ is an extension of $T_\alpha(M/M_1)$ by a quotient of $T_\alpha M_1$. Since $T_\alpha M$ is $\h$-semisimple, and a quotient of a highest weight module is still highest weight, by induction and Lemma \ref{extensionlemma} it suffices to show that if $M$ is highest weight of weight $(\lambda, \mu)$, then $T_\alpha M$ is highest weight of weight $(\lambda', s_\alpha(\mu))$ for some $\lambda' \in \h^*$.

Suppose $M$ is highest weight of weight $(\lambda, \mu)$, and let $v \in M$ be a highest weight generator of $M$. Then we claim that $f_\alpha^{-1} \fbar_\alpha^{-1} \otimes v$ is highest weight and generates $T_\alpha M$. We now verify this element is highest weight of weight $(s_\alpha(\lambda) - 2\alpha, s_\alpha(\mu))$.

Let $\beta \in \Phi^+ \backslash \{\alpha\}$. Then using relations $(5.9)$-$(5.12)$ in Lemma \ref{relationslemma} and the fact $v$ is highest weight: 
\begin{align*}
e_\beta \cdot_\alpha (f_\alpha^{-1} \fbar_\alpha^{-1} \otimes v) =& (e_{s_\alpha(\beta)} f_\alpha^{-1} \fbar_\alpha^{-1}) \otimes v \\
=& (f_\alpha^{-1} e_{s_\alpha(\beta)} - k_1 f_\alpha^{-2} e_{s_\alpha(\beta) - \alpha} + k_2 f_\alpha^{-3} e_{s_\alpha(\beta) - 2\alpha} - k_3 f_\alpha^{-4} e_{s_\alpha(\beta) - 3\alpha}) \fbar_\alpha^{-1} \otimes v\\
=& (f_\alpha^{-1} \fbar_\alpha^{-1} e_{s_\alpha(\beta)} - (l_1 f_\alpha^{-2} \fbar_\alpha^{-1} + l_2 f_\alpha^{-1} \fbar_\alpha^{-2}) e_{s_\alpha(\beta) - \alpha} \\
& + (l_3 f_\alpha^{-3} \fbar_\alpha^{-1} + l_4 f_\alpha^{-2} \fbar_\alpha^{-2}) e_{s_\alpha(\beta) - 2\alpha} \\
& - (l_5 f_\alpha^{-4} \fbar_\alpha^{-1} + l_6 f_\alpha^{-3} \fbar_\alpha^{-2})  e_{s_\alpha(\beta) - 3\alpha} ) \otimes v\\
=& 0\\
\overline{e}_\beta \cdot_\alpha (f_\alpha^{-1} \fbar_\alpha^{-1} \otimes v) &= (\overline{e}_{s_\alpha(\beta)} f_\alpha^{-1} \fbar_\alpha^{-1}) \otimes v \\
=& (f_\alpha^{-1} \overline{e}_{s_\alpha(\beta)} - k_1 f_\alpha^{-2} \overline{e}_{s_\alpha(\beta) - \alpha} + k_2 f_\alpha^{-3} \overline{e}_{s_\alpha(\beta) - 2\alpha} - k_3 f_\alpha^{-4} \overline{e}_{s_\alpha(\beta) - 3\alpha}) \fbar_\alpha^{-1} \otimes v\\
=& (f_\alpha^{-1} \fbar_\alpha^{-1} \overline{e}_{s_\alpha(\beta)} - k_1 f_\alpha^{-2} \fbar_\alpha^{-1} \overline{e}_{s_\alpha(\beta) - \alpha} + k_2 f_\alpha^{-3} \fbar_\alpha^{-1} \overline{e}_{s_\alpha(\beta) - 2\alpha} \\
&- k_3 f_\alpha^{-4} \fbar_\alpha^{-1} \overline{e}_{s_\alpha(\beta) - 3\alpha}) \otimes v\\
=& 0
\end{align*}
for some $k_i, l_i \in \C$. We also have:
\begin{align*}
e_\alpha \cdot_\alpha (f_\alpha^{-1} \fbar_\alpha^{-1} \otimes v) &= (f_\alpha f_\alpha^{-1} \fbar_\alpha^{-1}) \otimes v  = 0\\
\overline{e}_\alpha \cdot_\alpha (f_\alpha^{-1} \fbar_\alpha^{-1} \otimes v) &= (\fbar_\alpha f_\alpha^{-1} \fbar_\alpha^{-1}) \otimes v = 0
\end{align*}
and for any $h \in \h$, by relations $(5.7)$ and $(5.8)$ in Lemma \ref{relationslemma} we have:
\begin{align*}
\overline{h} \cdot_\alpha (f_\alpha^{-1} \fbar_\alpha^{-1} \otimes v) &= (\overline{\phi_\alpha(h)} f_\alpha^{-1} \fbar_\alpha^{-1} \otimes v) \\
&= f_\alpha^{-1} \overline{\phi_\alpha(h)} \fbar_\alpha^{-1} \otimes v + f_\alpha^{-2} \overline{\phi_\alpha(h)} \otimes v \\
&= f_\alpha^{-1} \fbar_\alpha^{-1} \otimes (\overline{\phi_\alpha(h)} \cdot v) \\
&= f_\alpha^{-1} \fbar_\alpha^{-1} \otimes s_\alpha(\mu)(h) v \\
&= s_\alpha(\mu)(h) (f_\alpha^{-1} \fbar_\alpha^{-1} \otimes v)
\end{align*}

Finally, by calculation (5.15) earlier showing that if $w \in M^{\lambda'}$ then $f_\alpha^{-n} \fbar_\alpha^{-m} \otimes w \in (T_\alpha M)^{s_\alpha(\lambda') - (n+m)\alpha}$, we see that $f_\alpha^{-1} \fbar_\alpha^{-1} \otimes v \in (T_\alpha M)^{s_\alpha(\lambda) - 2\alpha}$, so $f_\alpha^{-1} \fbar_\alpha^{-1} \otimes v$ is indeed highest weight of weight $(s_\alpha(\lambda) - 2\alpha, s_\alpha(\mu))$ as claimed.

To see $f_\alpha^{-1} \fbar_\alpha^{-1} \otimes v$ generates $T_\alpha M$, we observe that for any $v' \in M$, $v' = u \cdot v$ for some $u \in U$, and so for any $n, m > 0$, we have:
\[f_\alpha^{-n} \fbar_\alpha^{-m} \otimes v' = f_\alpha^{-n} \fbar_\alpha^{-m} \otimes (u \cdot v) = f_\alpha^{-n} \fbar_\alpha^{-m} u \otimes v = \sum u_i f_\alpha^{-n_i} \fbar_\alpha^{-m_i} \otimes v = \sum u_i \cdot (f_\alpha^{-n_i} \fbar_\alpha^{-m_i} \otimes v)\]
for some $n_i, m_i > 0$ and $u_i \in U$. Hence, since $T_\alpha M$ is spanned by $\{f_\alpha^{-n} \fbar_\alpha^{-m} \otimes v' : n,m > 0, v' \in M\}$, it is enough to show that for any $n, m >0$, the element $f_\alpha^{-n} \fbar_\alpha^{-m} \otimes v$ lies in the submodule of $T_\alpha M$ generated by $f_\alpha^{-1} \fbar_\alpha^{-1} \otimes v$. Consider $f_\alpha \cdot_\alpha (f_\alpha^{-1} \fbar_\alpha^{-1} \otimes v) = (e_\alpha f_\alpha^{-n} \fbar_\alpha^{-m}) \otimes v$ and $\fbar_\alpha \cdot_\alpha (f_\alpha^{-n} \fbar_\alpha^{-m} \otimes v) = (\overline{e}_\alpha f_\alpha^{-n} \fbar_\alpha^{-m}) \otimes v$. By relations $(5.1)$-$(5.4)$, we have:
\begin{align*}
(e_\alpha f_\alpha^{-n} \fbar_\alpha^{-m}) \otimes v &= (e_\alpha \fbar_\alpha^{-m} f_\alpha^{-n}) \otimes v \\
&= (\fbar_\alpha^{-m} e_\alpha f_\alpha^{-n} - m \fbar_\alpha^{-m - 1} \overline{h}_\alpha f_\alpha^{-n}) \otimes v \\
&= (\fbar_\alpha^{-m} e_\alpha f_\alpha^{-n} - m \mu(h_\alpha) \fbar_\alpha^{-m - 1} f_\alpha^{-n} - 2mn \fbar_\alpha^{-m} f_\alpha^{-n - 1}) \otimes v \\
&= (k - 2mn) f_\alpha^{-n - 1} \fbar_\alpha^{-m} \otimes v - m \mu(h_\alpha) f_\alpha^{-n} \fbar_\alpha^{-m - 1} \otimes v,
\end{align*}
for some $k \in \C$. In particular, the coefficient of $f_\alpha^{-n} \fbar_\alpha^{-m-1} \otimes v$ is always non-zero. We also have:
\begin{align*}
(\overline{e}_\alpha f_\alpha^{-n} \fbar_\alpha^{-m}) \otimes v &= (f_\alpha^{-n} \overline{e}_\alpha \fbar_\alpha^{-m} - 2n f_\alpha^{-n - 2} \fbar_\alpha^{-m+1} -  n f_\alpha^{-n-1} \fbar_\alpha^{-m} \overline{h}_\alpha  - n(n-1) f_\alpha^{-n - 2} \fbar_\alpha^{-m+1}) \otimes v \\
&= -n \mu(h_\alpha) f_\alpha^{-n - 1} \fbar_\alpha^{-m} \otimes v - n(n+1) f_\alpha^{-n - 2} \fbar_\alpha^{-m+1} \otimes v.
\end{align*}
In particular, the coefficients of $f_\alpha^{-n-1} \fbar_\alpha^{-m} \otimes v$ and $f_\alpha^{-n-2} \fbar_\alpha^{-m+1} \otimes v$ are both non-zero. Let $N \subseteq T_\alpha M$ be the submodule generated by $f_\alpha^{-1} \fbar_\alpha^{-1} \otimes v$. Then by the second calculation with $m = 1$, we see that $f_\alpha^{-n} \fbar_\alpha^{-1} \otimes v \in N$ for all $n > 0$. By the first calculation, we that if $f_\alpha^{-n'} \fbar_\alpha^{-m} \otimes v \in N$ for all $n' > 0$, then $f_\alpha^{-n} \fbar_\alpha^{-m-1} \otimes v \in N$ for all $n > 0$, which completes the proof that $N = T_\alpha M$ by induction.
\end{proof}

We now aim to prove part (b). First, we prove the following results on the invariants of a module $M \in \calO_\epsilon^\mu$ with respect to a certain subalgebra of $\g$, and then prove a result on the tensor product $S_\alpha \otimes_U M$. Here we let ${(\mathfrak{sl}_2)}_\alpha = \langle e_\alpha, h_\alpha, f_\alpha \rangle \subseteq \g$. 

\begin{lemma}
\label{freemodulelemma}
Let $\mu \in \h^*$ such that $\mu(h_\alpha) \neq 0$, let $M \in \calO_\epsilon^\mu$, and let $\mathfrak{a} = \langle f_\alpha, \fbar_\alpha \rangle \subseteq \g_\epsilon$. Then $M$ is free as a $U(\mathfrak{a})$-module, and any vector space basis of $M^{\langle e_\alpha, \overline{e}_\alpha \rangle}$ is a free generating set.
\end{lemma}

\begin{proof}
It clearly suffices to only consider the case where $M$ is indecomposable.

Consider the case where $\g = \mathfrak{sl}_2$ and $M$ is a Verma module. In this case, $M^{\langle e_\alpha, \overline{e}_\alpha \rangle}$ is one dimensional, spanned by any highest weight vector of $M$. But by definition $M$ is generated freely by this as a $U((\n^-)_\epsilon)$-module, and in this case $\mathfrak{a} = \n^-$ so the lemma holds.

If $\g = \mathfrak{sl}_2$ and $M$ is indecomposable but not a Verma module, then by earlier results $M$ has a filtration $0 = M_0 \subseteq M_1 \subseteq \dots \subseteq M_{k-1} \subseteq M_k = M$ such that each quotient is a Verma module. To show any basis generates $M$ as a $U(\mathfrak{a})$-module, it suffices to check one choice.

Choose $\Omega = \Psi \cup \{v\}$, where $\Psi$ is a basis of $M_{k-1}^{\langle e_\alpha, \overline{e}_\alpha \rangle}$ and $v \in M^{\langle e_\alpha, \overline{e}_\alpha \rangle} \backslash M_{k-1}$. Let $M' \subseteq M$ be the $U(\mathfrak{a})$-submodule generated by $\Omega$. Then observe that $M/M_{k-1}$ is a Verma module, with highest weight generator $v + M_{k-1}$. Hence for any $m \in M$, there is some $m' \in M'$ such that $m - m' \in M_{k-1}$. But $\Psi \subseteq \Omega$, and by induction $\Psi$ generates $M_{k-1}$ as a $U(\mathfrak{a})$-module, so $M_{k-1} \subseteq M'$. Hence $m - m' \in M'$, and so $m \in M'$ for any $m \in M$, i.e. $M = M'$.

Now we show that $\Omega$ generates $M$ freely, i.e. that the set $\Omega' = \{f_\alpha^i \fbar_\alpha^j v : i, j \geq 0, v \in \Omega\}$ is linearly independent. We already know that $\Omega'$ spans $M$, so it is enough to show that the number of $f_\alpha^i \fbar_\alpha^j b$ of weight $\lambda'$ is equal to the dimension of $M^{\lambda'}$ for any $\lambda' \in \h^*$. Now, every quotient $M_i/M_{i-1}$ is isomorphic to the same Verma module $M_{\lambda, \mu}$ and hence all $v \in \Omega$ have weight $\lambda$ and also $\dim(M^{\lambda'}) = k \dim(M_{\lambda, \mu}^{\lambda'})$. These two facts together imply that the number of $f_\alpha^i \fbar_\alpha^j v$ of weight $\lambda'$ is equal to the dimension of $M^{\lambda'}$, so $\Omega$ generates $M$ freely and the lemma holds in the case $\g = \mathfrak{sl}_2$.

Finally, we deal with the case where $\g$ is any reductive Lie algebra. Let $\Upsilon$ be a basis of $M^{\langle e_\alpha, \overline{e}_\alpha \rangle}$. Let $m \in M$ and let $N$ be the $((\mathfrak{sl}_2)_\alpha)_\epsilon$-submodule of $M$ generated by $m$. It is easy to check that $N \in \calO^{\mu(h_\alpha)}((\mathfrak{sl}_2)_\alpha)$, since all the axioms except finite generation follow from the fact $M \in \calO^\mu$, while by definition $N$ is generated by one element. Hence we can apply the lemma in the $\mathfrak{sl}_2$ case to see that $m = \sum_k f_\alpha^{i_k} \fbar_\alpha^{j_k} m_k$, for some $i_k, j_k \geq 0$ and $m_k \in N^{\langle e_\alpha, \overline{e}_\alpha \rangle} \subseteq M^{\langle e_\alpha, \overline{e}_\alpha \rangle} = \vspan \Upsilon$, so $m$ is in the $U(\mathfrak{a})$-module generated by $\Upsilon$. 

Suppose $\Upsilon$ does not generate $M$ freely as a $U(\mathfrak{a})$-module. Then $\sum_k f_\alpha^{i_k} \fbar_\alpha^{j_k} b_k = 0$ for some $i_k, j_k \geq 0$ and $b_k \in \Upsilon$. But since this sum is finite, we can let $N$ be the $((\mathfrak{sl}_2)_\alpha)_\epsilon$-module generated by $\{b_1, \dots, b_n\}$ and obtain a contradiction to the $\mathfrak{sl}_2$ case. Hence the lemma holds.
\end{proof}

\begin{corollary}
\label{splittingcor}
Let $\mu$ be such that $\mu(h_\alpha) \neq 0$, let $M \in \calO_\epsilon^\mu$, and let $N_1$ be a finitely generated $((\mathfrak{sl}_2)_\alpha)_\epsilon$-submodule of $M$. Then there exists a $U(\mathfrak{a})$-submodule $N_2$ of $M$ such that $M = N_1 \oplus N_2$, where again $\mathfrak{a} = \langle f_\alpha, \fbar_\alpha \rangle$.
\end{corollary}

\begin{proof}
Let $\Psi$ be a basis for $N_1^{\langle e_\alpha, \overline{e}_\alpha \rangle}$, which freely generates $N_1$ as a $U(\mathfrak{a})$-module by the previous lemma. Since $N_1^{\langle e_\alpha, \overline{e}_\alpha \rangle} \subseteq M^{\langle e_\alpha, \overline{e}_\alpha \rangle}$, we can extend $\Psi$ to a basis $\Omega$ for $M^{\langle e_\alpha, \overline{e}_\alpha \rangle}$, so if we let $N_2$ be the $U(\mathfrak{a})$-module generated by $\Omega \backslash \Psi$ then by the previous lemma we have $M = N_1 \oplus N_2$.
\end{proof}

We now recall the following general result about tensor products. Let $R$ be a ring, and let $M$ and $N$ be right and left $R$-modules respectively, and let $V$ be a $\C$-vector space. We say a map $\varphi: M \times N \rightarrow V$ is $R$-balanced if it is $\C$-bilinear and for any $m \in M$, $n \in N$, and $r \in R$, we have $\varphi(m \cdot r, n) = \varphi(m, r \cdot n)$. We then have the following standard result on the tensor product $M \otimes_R N$:
\begin{lemma}
\label{balancedmapslemma}
The element $m \otimes n$ of $M \otimes_R N$ is zero if and only if for any vector space $V$ and $R$-balanced map $\varphi: M \times N \rightarrow V$, we have that $\varphi(m, n) = 0$.
\end{lemma}

We observe that if $\mathfrak{a} = \langle f_\alpha, \fbar_\alpha \rangle$ as above and $A = \vspan \{f_\alpha^{-n} \fbar_\alpha^{-m}: n, m > 0 \} \subseteq S_\alpha$, which is a $U(\mathfrak{a})$-$U(\mathfrak{a})$-subbimodule of $S_\alpha$, then we have (by considering the two bases of $S_\alpha$ given when we defined it) that as $U$-$U(\mathfrak{a})$-bimodules:
\[S_\alpha \cong U \otimes_{U(\mathfrak{a})} A,\]
and as $U(\mathfrak{a})$-$U$ bimodules:
\[S_\alpha \cong A \otimes_{U(\mathfrak{a})} U.\]
In particular, for any left $U$-module $M$, we have that as left $U(\mathfrak{a})$-modules:
\begin{align*}
S_\alpha \otimes_U M \cong (A \otimes_{U(\mathfrak{a})} U) \otimes_U M \cong A \otimes_{U(\mathfrak{a})} M.
\end{align*}

The following lemma, which we will prove using Corollary \ref{splittingcor}, Lemma \ref{balancedmapslemma} and the above observation, is very useful in the proof of part (b) of Lemma \ref{maintwistinglemma}:

\begin{lemma}
\label{tensorzerolemma}
Let $\mu \in \h^*$ be such that $\mu(h_\alpha) \neq 0$ and let $M \in \calO_\epsilon^\mu$. Let $n \in M^\lambda \backslash \{0\}$ for some $\lambda \in \h^*$. Then the following are equivalent:
\begin{enumerate}
\item[(a)] In $T_\alpha M$, we have $f_\alpha^{-i} \fbar_\alpha^{-j} \otimes n = 0$.

\item[(b)] For any vector space $V$ and $U(\mathfrak{a})$-balanced map $\varphi: A \times M \rightarrow V$ we have $\varphi(f_\alpha^{-i} \fbar_\alpha^{-j}, n) = 0$.

\item[(c)] There exist $n_1, n_2 \in M$ such that $n = f_\alpha^i \cdot n_1 + \fbar_\alpha^j \cdot n_2$.
\end{enumerate}
\end{lemma}

\begin{proof}
We first note that (a) and (b) are equivalent by the above observation and Lemma \ref{balancedmapslemma}. Also, if $n = f_\alpha^i \cdot n_1 + \fbar_\alpha^j \cdot n_2$ for some $n_1, n_2 \in M$, then $f_\alpha^{-i} \fbar_\alpha^{-j} \otimes n = f_\alpha^{-i} \fbar_\alpha^{-j} \otimes f_\alpha^i \cdot n_1 + f_\alpha^{-i} \fbar_\alpha^{-j} \otimes \fbar_\alpha^j \cdot n_2 = f_\alpha^0 \fbar_\alpha^{-j} \otimes n_1 + f_\alpha^{-i} \fbar_\alpha^0 \otimes n_2 = 0$, so certainly (c) implies (a), and in fact this is true for any $M \in U$-mod. To show (b) implies (c), we suppose $n \in M$ cannot be written in the form $f_\alpha^i \cdot n_1 + \fbar_\alpha^j \cdot n_2$. We seek a $U(\mathfrak{a})$-balanced map $\varphi: A \times M \rightarrow V$ such that $\varphi(f_\alpha^{-i} \fbar_\alpha^{-j}, n) \neq 0$. It suffices to consider the case where $M$ is indecomposable.

First consider the case where $\g = \mathfrak{sl}_2$ and $M$ is a Verma module of weight $(\lambda, \mu)$, where $\mu \neq 0$. Let $n \in M$. Then $n = \sum_{a, b} k_{a, b} f_\alpha^a \fbar_\alpha^b \otimes 1$, where all but finitely many of the $k_{a, b}$ are zero. Observe that $n$ is of the form $n = f_\alpha^i \cdot n_1 + \fbar_\alpha^j \cdot n_2$ for some $n_1, n_2 \in M$ if and only if $k_{a, b} = 0$ for all $a < i$ and $b < j$.

Now suppose $n$ is not of the form $n = f_\alpha^i \cdot n_1 + \fbar_\alpha^j \cdot n_2$. Then pick some $0 < a \leq i$ and $0 < b \leq j$ such that $k_{i-a, j-b} \neq 0$. We can define a $U(\mathfrak{a})$-balanced map $\varphi : A \times M \rightarrow \C$ by setting:
\[\varphi(f_\alpha^{-c} \fbar_\alpha^{-d}, f_\alpha^r \fbar_\alpha^s \otimes 1) = 
\begin{cases}
    1               & \text{if } c - r = a \text{ and } d - s = b\\
    0               & \text{otherwise}
\end{cases}
\]
and extending linearly. By construction $\varphi$ is a $U(\mathfrak{a})$-balanced map with $\varphi(f_\alpha^{-i} \fbar_\alpha^{-j}, n) \neq 0$.

Now we consider the case where $\g = \mathfrak{sl}_2$ and $M$ is indecomposable. If $M$ is a Verma module, we are done by the above. If not, let $0 = M_0 \subseteq M_1 \subseteq \dots \subseteq M_{k-1} \subseteq M_k = M$ be a filtration of $M$ such that each section is a highest weight module. In fact, by \cite[Theorem 7.1]{W}, in this case each Verma module is simple, so each section is in fact a Verma module. Let $n \in M$ be such that $n$ cannot be written as $n = f_\alpha^i \cdot n_1 + \fbar_\alpha^j \cdot n_2$.
Consider the quotient $M/M_1$. We have one of two situations:

(1) If $n + M_1$ cannot be written as $f_\alpha^i \cdot n_1 + \fbar_\alpha^j \cdot n_2 + M_1$, then by induction on $k$ we can find a $U(\mathfrak{a})$-balanced map $\varphi: A \times (M/M_1) \rightarrow \C$ such that $\varphi(f_\alpha^{-i} \fbar_\alpha^{-j}, n + M_1) \neq 0$. We can then lift this to a $U(\mathfrak{a})$-balanced map $\overline{\varphi}: A \times M \rightarrow \C$ by setting $\overline{\varphi}(u, m) = \varphi(u, m + M_1)$, which clearly satisfies $\overline{\varphi}(f_\alpha^{-i} \fbar_\alpha^{-j}, n) \neq 0$.

(2) If $n + M_1 = f_\alpha^i \cdot n'_1 + \fbar_\alpha^j \cdot n'_2 + M_1$, then let $v = n - f_\alpha^i \cdot n'_1 - \fbar_\alpha^j \cdot n'_2 \in M_1$. Now $v$ cannot be written as $f_\alpha^i \cdot v_1 + \fbar_\alpha^j \cdot v_2$ else $n$ could be written as $f_\alpha^i \cdot n_1 + \fbar_\alpha^j \cdot n_2$ for some $n_1, n_2$, and $M_1$ is a Verma module, so we have already shown there exists a $U(\mathfrak{a})$-balanced map $\varphi: A \times M_1 \rightarrow \C$ such that $\varphi(f_\alpha^{-i}, \fbar_\alpha^{-j}, v) \neq 0$. Now by Corollary \ref{splittingcor}, there exists a $U(\mathfrak{a})$-submodule $M'$ of $M$ such that $M = M_1 \oplus M'$. Hence 
\[(\varphi \oplus 0): A \times M \cong A \times (M_1 \oplus M') \cong (A \times M_1) \oplus (A \times M') \rightarrow \C\]
is a $U(\mathfrak{a})$-balanced map such that: 
\[(\varphi \oplus 0)(f_\alpha^{-i} \fbar_\alpha^{-j}, n) = (\varphi \oplus 0)(f_\alpha^{-i} \fbar_\alpha^{-j}, v + f_\alpha^i \cdot n'_1 + \fbar_\alpha^j \cdot n'_2) = \varphi(f_\alpha^{-i} \fbar_\alpha^{-j}, v) \neq 0\]

Finally, let $\g$ be an arbitrary reductive Lie algebra, let $M \in \calO_\epsilon^\mu$, and let $n \in M$ be such that $n$ cannot be written as $n = f_\alpha^i \cdot n_1 + \fbar_\alpha^j \cdot n_2$. Let $N_1$ be the $((\mathfrak{sl}_2)_\alpha)_\epsilon$-module generated by $n$, and let $N_2$ be a $U(\mathfrak{a})$-module such that $M = N_1 \oplus N_2$ as in Corollary \ref{splittingcor}. By the $\mathfrak{sl}_2$ case, we have a $U(\mathfrak{a})$-balanced map $\varphi: A \times N_1 \rightarrow \C$ such that $\varphi(f_\alpha^{-i} \fbar_\alpha^{-j}, n) \neq 0$. Then
\[(\varphi \oplus 0): A \times M \cong A \times (N_1 \oplus N_2) \cong (A \times N_1) \oplus (A \times N_2) \rightarrow \C\]
is a $U(\mathfrak{a})$-balanced map such that $(\varphi \oplus 0)(f_\alpha^{-i} \fbar_\alpha^{-j}, n) = \varphi(f_\alpha^{-i} \fbar_\alpha^{-j}, n) \neq 0$ as required.
\end{proof}

\begin{corollary}
\label{tensorzerocorollary}
Let $M \in \calO_\epsilon^\mu(\g)$ and let $\lambda \in \h^*$. Then there is some $n \in \mathbb{N}$ such that $M^{\lambda + n \alpha} = 0$. Furthermore for any such $n$ and any $m \in M^\lambda$, we have $f_\alpha^{-n} \fbar_\alpha^{-n} \otimes m = 0$ if and only if $m = 0$.
\end{corollary}

\begin{proof}
The existence of some $n$ such that $M^{\lambda + n \alpha} = 0$ follows from Lemma \ref{fdweightspaces}(b). By Lemma \ref{tensorzerolemma}, we have that $f_\alpha^{-n} \fbar_\alpha^{-n} \otimes m = 0$ if and only if there exist $m_1, m_2 \in M$ such that $m = f_\alpha^n \cdot m_1 + \fbar_\alpha^n \cdot m_2$. Replacing $m_1$ and $m_2$ with their projections into the $\lambda + n\alpha$ weight space if necessary, we may assume $m_1, m_2 \in M^{\lambda + n \alpha} = 0$, so $m = 0$.
\end{proof}

Finally, we need one more lemma:

\begin{lemma}
\label{elementformlemma}
Let $M \in U$-mod. Then any element of $S_\alpha \otimes M$ can be written in the form $f_\alpha^{-k} \fbar_\alpha^{-l} \otimes m$, for some $k, l > 0$ and $m \in M$.
\end{lemma}

\begin{proof}
Certainly any element of $S_\alpha \otimes M$ can be written as a sum of elements of the form $s \otimes m$ for some $s \in S_\alpha$ and $m \in M$. We have that any $s \in S_\alpha$ is the sum of elements of the form $f_\alpha^{-i} \fbar_\alpha^{-j} u$, where $u = v(\mathbf{k}, \mathbf{l}, \mathbf{m}, \mathbf{n})$ for some $\mathbf{k}, \mathbf{l} \in \Z_{\geq 0}^I$, $\mathbf{m}, \mathbf{n} \in \Z_{\geq 0}^J$, and $f_\alpha^{-i} \fbar_\alpha^{-j} u \otimes m = f_\alpha^{-i} \fbar_\alpha^{-j} \otimes (u \cdot m)$ for any $m \in M$, so any element of $S_\alpha \otimes M$ can be written as the sum of elements of the form $f_\alpha^{-k} \fbar_\alpha^{-l} \otimes m$.

Now let $x \in S_\alpha \otimes M$, and write $x = \sum_a f_\alpha^{-k_a} \fbar_\alpha^{-l_a} \otimes m_a$. Let $k = \max\{k_a\}, l = \max\{l_a\}$. Then:
\begin{align*}
x = \sum_a (f_\alpha^{-k_a} \fbar_\alpha^{-l_a} \otimes m_a) &= \sum_a (f_\alpha^{-k} \fbar_\alpha^{-l} \otimes f_\alpha^{k - k_a} \fbar_\alpha^{l - l_a} m_a) \\
&= f_\alpha^{-k} \fbar_\alpha^{-l} \otimes \sum_a (f_\alpha^{k - k_a} \fbar_\alpha^{l - l_a} m_a)
\end{align*}
\end{proof}

This now allows us to prove part (b) of Lemma \ref{maintwistinglemma}:

\begin{proof}[Proof of \ref{maintwistinglemma}(b)]

First observe that $\psi_M$ is certainly a homomorphism since for any $m \in M$, $u \in U$ and $s \in S_\alpha$, we have $(u \cdot \psi_M(m))(s) = \psi_M(m)(s \cdot u) = (s \cdot u) \otimes m = s \otimes (u \cdot m) = \psi_M(u \cdot m)(s)$.

To see that $\psi_M$ is always injective, let $M \in \calO_\epsilon^\mu$ and let $m \in M$ be a non-zero weight vector of weight $\lambda$ say. Then, applying Corollary \ref{tensorzerocorollary}, there exists $n \in \mathbb{N}$ such that $M^{\lambda + n \alpha} = 0$, and $\psi_M(m)(f_\alpha^{-n} \fbar_\alpha^{-n}) = f_\alpha^{-n} \fbar_\alpha^{-n} \otimes m = 0$ if and only if $m = 0$. Hence $\psi_M(m) = 0$ if and only if $m = 0$. Since $M$ and $G_\alpha T_\alpha M$ are both $\h$-semisimple, this suffices to show that $\psi_M$ is injective.

We now show that $\psi_M$ is surjective. Let $g \in G_\alpha T_\alpha M$. This $g$ is a map $S_\alpha \rightarrow S_\alpha \otimes M$, and we may assume $g$ is a weight element of $G_\alpha T_\alpha M$, of weight $\lambda$ say. We first use Lemma \ref{elementformlemma} to see that, for all $i, j > 0$ we can find $k_{i, j}, l_{i, j} > 0$, $m_{i, j} \in M$ such that $g(f_\alpha^{-i} \fbar_\alpha^{-j}) = f_\alpha^{-k_{i, j}} \fbar_\alpha^{-l_{i, j}} \otimes m_{i, j}$.

If $i > k_{i, j}$, then we may replace $m_{i, j}$ with $f_\alpha^{i - k_{i, j}} \cdot m_{i, j}$ and set $k_{i, j} = i$, since
\[f_\alpha^{-k_{i, j}} \fbar_\alpha^{-l_{i, j}} \otimes m_{i, j} = f_\alpha^{-i} \fbar_\alpha^{-l_{i, j}} \otimes (f_\alpha^{i - k_{i, j}} \cdot m_{i, j})\]
and hence we may assume that $i \leq k_{i, j}$. Similarly, we may assume that $j \leq l_{i, j}$. On the other hand, suppose $i < k_{i, j}$. Then since $g$ is a $U(\g_\epsilon)$-homomorphism, we have:
\[0 = f_\alpha^i \cdot g(f_\alpha^{-i} \fbar_\alpha^{-j}) = f_\alpha^{-k_{i, j} + i} \fbar_\alpha^{-l_{i, j}} \otimes m_{i, j}\]
By Lemma \ref{tensorzerolemma} we have that $m_{i, j} = f_\alpha^{k_{i, j} - i} \cdot m_1 + \fbar_\alpha^{l_{i, j}} \cdot m_2$ for some $m_1, m_2 \in M$. Hence:
\[g(f_\alpha^{-i} \fbar_\alpha^{-j}) = f_\alpha^{-k_{i, j}} \fbar_\alpha^{-l_{i, j}} \otimes (f_\alpha^{k_{i, j} - i} \cdot m_1 + \fbar_\alpha^{l_{i, j}} \cdot m_2) = f_\alpha^{-i} \fbar_\alpha^{-l_{i, j}} \otimes m_1\]
so, replacing $m_{i, j}$ with $m_1$, we may assume that $k_{i, j} = i$, and similarly we may assume that $l_{i, j} = j$.

Since $g$ was assumed to have weight $\lambda$, we have that for any $h \in \h$,
\[\lambda(h) g(f_\alpha^{-i} \fbar_\alpha^{-j}) = (h \cdot g)(f_\alpha^{-i} \fbar_\alpha^{-j}) = g(f_\alpha^{-i} \fbar_\alpha^{-j} h) = g((h - (i + j)\alpha(h)) f_\alpha^{-i} \fbar_\alpha^{-j})\]
\[= (h - (i + j)\alpha(h)) \cdot g(f_\alpha^{-i} \fbar_\alpha^{-j})\]
so $g(f_\alpha^{-i} \fbar_\alpha^{-j}) = f_\alpha^{-i} \fbar_\alpha^{-j} \otimes m_{i, j}$ has weight $\lambda + (i+j)\alpha$. But by a similar calculation as (5.15) (but using the untwisted action on $S_\alpha \otimes M$ rather than the twisted action used there), if $m \in M$ has weight $\lambda'$, then $f_\alpha^{-i} \fbar_\alpha^{-j} \otimes m$ has weight $\lambda' + (i+j)\alpha$ in $S_\alpha \otimes M$. Hence the $m_{i, j}$ all have weight $\lambda$ in $M$.

Now we choose $n$ such that $M^{\lambda + n \alpha} = 0$. Then for any $i, j \geq n$, we have
\[f_\alpha^{-n} \fbar_\alpha^{-n} \otimes m_{n, n} = g(f_\alpha^{-n} \fbar_\alpha^{-n}) = f_\alpha^{i - n} \fbar_\alpha^{j - n} \cdot g(f_\alpha^{-i} \fbar_\alpha^{-j}) = f_\alpha^{-n} \fbar_\alpha^{-n} \otimes m_{i, j}\]
so $f_\alpha^{-n} \fbar_\alpha^{-n} \otimes (m_{n, n} - m_{i, j}) = 0$. But, again applying Corollary \ref{tensorzerocorollary}, we must have $m_{n, n} - m_{i, j} = 0$, i.e. $m_{n, n} = m_{i, j}$. Let $m = m_{n, n}$. Then we argue that $f_\alpha^{-i} \fbar_\alpha^{-j} \otimes m = f_\alpha^{-i} \fbar_\alpha^{-j} \otimes m_{i, j}$ for all $i, j > 0$: if $i, j \geq n$ we have just seen this, and if not then we can use that fact that \[f_\alpha^{-i + 1} \fbar_\alpha^{-j} \otimes m_{i, j} = f_\alpha \cdot g(f_\alpha^{-i} \fbar_\alpha^{-j}) = g(f_\alpha^{-i+1} \fbar_\alpha^{-j}) = f_\alpha^{-i+1} \fbar_\alpha^{-j} \otimes m_{i-1, j}\]
(along with a similar result relating $m_{i, j}$ and $m_{i, j-1}$).

We now claim that $g = \psi_M(m)$ for this $m$ chosen above. For any $i, j > 0$, $u \in U$, let $i', j' > 0$, $u' \in U$ be such that $f_\alpha^{-i} \fbar_\alpha^{-j} u = u' f_\alpha^{-i'} \fbar_\alpha^{-j'}$. Then
\[g(f_\alpha^{-i} \fbar_\alpha^{-j} u) = g(u' f_\alpha^{-i'} \fbar_\alpha^{-j'}) = u' f_\alpha^{-i'} \fbar_\alpha^{-j'} \otimes m = f_\alpha^{-i} \fbar_\alpha^{-j} u \otimes m = \psi_M(m)(f_\alpha^{-i} \fbar_\alpha^{-j} u)\]
Since $\{f_\alpha^{-i} \fbar_\alpha^{-j} u: i, j > 0, u \in U\}$ spans $S_\alpha$, we have $g = \psi_M(m)$ as required.
\end{proof}

Before we prove Lemma \ref{maintwistinglemma}(c), we show the following:

\begin{lemma}
\label{homextensionlemma}
Let $M$ be a $U$-module, and let $A \subseteq S_\alpha$ and $\mathfrak{a} \subseteq \g$ be as in Lemma \ref{tensorzerolemma}. Let $\varphi: A \rightarrow M$ be a $U(\mathfrak{a})$-homomorphism. Then $\varphi$ extends uniquely to a $U$-homomorphism $\varphi: S_\alpha \rightarrow M$.
\end{lemma}

\begin{proof}
Frobenius reciprocity states that for algebras $R \subseteq S$, an $R$-module $N$ and an $S$-module $M$, there is an isomorphism $\Hom_S(S \otimes_R N, M) \cong \Hom_R(N, M)$ given by the restriction map $\varphi \mapsto \varphi|_{1 \otimes R}$. Setting $R = U(\mathfrak{a})$, $S = U$, and $N = A$, the result follows from the earlier observation that $S_\alpha \cong U \otimes_{U(\mathfrak{a})} A$ as left $U$-modules.
\end{proof}

\begin{lemma}
\label{homconstruction}
Let $M \in \calO_\epsilon^\mu$ and let $\mathcal{I}$ be a subset of $\Z^2$ satisfying:

\begin{enumerate}
\item[(1)] $(i, j) \in\mathcal{I}$ whenever $i \leq 0$ or $j \leq 0$.

\item[(2)] If $(i, j) \in \mathcal{I}$, then $(i-1, j) \in \mathcal{I}$ and $(i, j-1) \in \mathcal{I}$.
\end{enumerate}
and let $\{m_{i, j} \in M : (i, j) \in \mathcal{I}\}$ be a collection of elements of $M$ satisfying:
\begin{enumerate}
\item[(i)] $m_{i, j} = 0$ whenever $i \leq 0$ or $j \leq 0$.

\item[(ii)] $e_\alpha \cdot m_{i, j} = m_{i-1, j}$ whenever $(i, j) \in \mathcal{I}$

\item[(iii)] $\overline{e}_\alpha \cdot m_{i, j} = m_{i, j-1}$ whenever $(i, j) \in \mathcal{I}$.
\end{enumerate}

Then there exists a $U(\mathfrak{a})$-homomorphism $\varphi: A \rightarrow \phi_\alpha^{-1}(M)$ such that $\varphi(f_\alpha^{-i} \fbar_\alpha^{-j}) = m_{i, j}$, which by Lemma \ref{homextensionlemma} extends to a $U$-homomorphism $\varphi: S_\alpha \rightarrow \phi_\alpha^{-1}(M)$. Moreover, if there exists $\lambda \in \h^*$ such that $\wt(m_{i, j}) = \lambda - (i+j)\alpha$ for all $i, j \in I$, then we can choose $\varphi$ to also be weight and hence in $G_\alpha M$.
\end{lemma}

\begin{proof}
We construct elements $m_{i, j}$ for $(i, j) \in \Z^2 \backslash \mathcal{I}$ inductively such that the $m_{i, j}$ satisfy conditions (i) - (iii) above for any $(i, j) \in \Z^2$. Then observe that (i) - (iii) ensure that setting $\varphi(f_\alpha^{-i} \fbar_\alpha^{-j}) = m_{i, j}$ defines a $U(\mathfrak{a})$-homomorphism $\varphi: A \rightarrow \phi_\alpha^{-1}(M)$, since it is enough to check that $m_{i-1, j} = \varphi(f_\alpha^{-i+1} \fbar_\alpha^{-j}) = f_\alpha \cdot_{\alpha^{-1}} \varphi(f_\alpha^{-i} \fbar_\alpha^{-j}) = f_\alpha \cdot_{\alpha^{-1}} m_{i, j} = e_\alpha \cdot m_{i, j}$ and a similar condition for $m_{i, j-1}$.

To construct such $m_{i, j}$, let $(i, j) \in \Z^2 \backslash \mathcal{I}$ be such that $i + j$ is minimal among elements of $\Z^2 \backslash \mathcal{I}$. Then in particular, $(i-1, j), (i, j-1) \in \mathcal{I}$. Let $N$ be the $((\mathfrak{sl}_2)_\alpha)_\epsilon$-module generated by $m_{i-1, j}$ and $m_{i, j-1}$, so $N \in \calO^{\mu(h_\alpha)}((\mathfrak{sl}_2)_\alpha)$. Then we use the following claim to construct $m_{i, j} \in N \subseteq M$ such that $e_\alpha \cdot m_{i, j} = m_{i-1, j}$ and $\overline{e}_\alpha \cdot m_{i, j} = m_{i, j-1}$.

\begin{claim}
Consider the maps:
\[N \overset{\theta_1}{\longrightarrow} N^2 \overset{\theta_2}{\longrightarrow} N\]
given by $\theta_1(x) = (e_\alpha \cdot x, \overline{e}_\alpha \cdot x)$ and $\theta_2(y, z) = \overline{e}_\alpha \cdot y - e_\alpha \cdot z$. Then $\ker(\theta_2) = \im(\theta_1)$. Furthermore, if $y, z$ are both in $M^\lambda \cap N$, then there exists $x \in M^{\lambda - \alpha}$ such that $\theta_1(x) = (y, z)$.
\end{claim}

\begin{proof}[Proof of claim:]
First observe that $\theta_2 \circ \theta_1 = 0$, so certainly $\ker(\theta_2) \supseteq \im(\theta_1)$. Hence we only need to show that $\im(\theta_1) \supseteq \ker(\theta_2)$.

We first deal with the case where $\g = \mathfrak{sl}_2$ and $N = M_{\lambda', \mu}$. In this case, we consider the restriction of these maps to certain weight spaces in the following way (for any $\lambda \in \h^*$):
\[N^\lambda \overset{\theta_1}{\longrightarrow} (N^{\lambda + \alpha})^2 \overset{\theta_2}{\longrightarrow} N^{\lambda + 2\alpha}\]
Now, either $\dim(N^{\lambda + \alpha}) = 0$, in which case $\im(\theta_1) = \ker(\theta_2)$ automatically, or the dimensions of these weight spaces satisfy $\dim(N^\lambda) = n + 1$, $\dim(N^{\lambda + \alpha}) = n$, and $\dim(N^{\lambda + 2\alpha}) = n - 1$. Hence by considering dimensions and the fact that $\ker(\theta_2) \supseteq \im(\theta_1)$, it is enough to show that $\theta_1$ is injective and $\theta_2$ is surjective. We can compute that $\overline{e}$ acts on the basis vectors $\overline{f}^i f^j \otimes 1$ by:
\[\overline{e} \cdot (\overline{f}^i f^j \otimes 1) = \mu j \overline{f}^i f^{j-1} \otimes 1 - j(j-1)\overline{f}^{i+1} f^{j-2} \otimes 1\]
so by considering $\theta_2(x, 0)$, we see $\theta_2$ is surjective. We also see that $\overline{e} \cdot v = 0$ if and only if $v = \overline{f}^i \otimes 1$. Since $\mu \neq 0$, we have $e \cdot \overline{f}^i \otimes 1 \neq 0$, so $\theta_1$ is injective as required.

If $\g = \mathfrak{sl}_2$ but $N$ is not a Verma module, let $0 = N_0 \subseteq N_1 \subseteq \dots \subseteq N_{k-1} \subseteq N_k = N$ be a filtration of $N$ such that each quotient is a Verma module. Then, given $y, z$ such that $e_\alpha \cdot z = \overline{e}_\alpha \cdot y$, in the quotient $N/N_{k-1} \cong M_{\lambda, \mu}$ we have that there exists $x \in N$ such that $e_\alpha \cdot x + N_{k-1} = y + N_{k-1}$ and $\overline{e}_\alpha \cdot x + N_{k-1} = z + N_{k-1}$. Hence $e_\alpha \cdot x - y, \overline{e}_\alpha \cdot x - z \in N_{k-1}$. But $\overline{e}_\alpha \cdot (e_\alpha \cdot x - y) = e_\alpha \cdot (\overline{e}_\alpha \cdot x - z)$, so by induction on $k$, there exists $x' \in N_{k-1}$ such that $e_\alpha \cdot x' = e_\alpha \cdot x - y$ and $\overline{e}_\alpha \cdot x' = \overline{e}_\alpha \cdot x - z$. Hence $e_\alpha \cdot (x-x') = y$ and $\overline{e}_\alpha \cdot (x-x') = z$, so $\im(\theta_1) \supseteq \ker(\theta_2)$.

Finally, if $y, z \in M^\lambda \cap N$, then given $x \in N \subseteq M$ such that $e_\alpha \cdot x = y$ and $\overline{e}_\alpha \cdot x = z$, by considering weight spaces we may replace $x$ with its component in the $\lambda - \alpha$ weight space and this still holds, proving the final part of the claim.
\end{proof}

Since $e_\alpha \cdot m_{i, j-1} = m_{i-1, j-1} = \overline{e}_\alpha \cdot m_{i-1, j}$, this claim then allows us to pick $m_{i, j}$ such that $e_\alpha \cdot m_{i, j} = m_{i-1, j}$ and $\overline{e}_\alpha \cdot m_{i, j} = m_{i, j-1}$, and if $m_{i-1, j}$ and $m_{i, j-1}$ both have weight $\lambda$, then we can choose $m_{i, j}$ to have weight $\lambda - \alpha$. Hence applying this inductively, we can construct $m_{i, j}$ satisfying conditions (i) - (iii).

Suppose there exists $\lambda \in \h^*$ such that $\wt(m_{i, j}) = \lambda - (i+j)\alpha$ for all $i, j \in \mathcal{I}$. Then by construction $m_{i, j}$ is weight for all $(i, j) \in \Z^2$, so by Corollary \ref{weightfunctionscor} this $\varphi$ we have constructed is a weight vector.
\end{proof}

We can now prove Lemma \ref{maintwistinglemma}(c) and (d).

\begin{proof}[Proof of \ref{maintwistinglemma}(c)]

First we show $\epsilon_N$ is a homomorphism. Let $u \in U$, let $s \in S_\alpha$, and let $g \in G_\alpha N$. Then
\[u \cdot \epsilon_N(s \otimes g) = u \cdot g(s) = \phi_\alpha^{-1} \phi_\alpha (u) \cdot g(s) = g(\phi_\alpha(u) \cdot s) = \epsilon_N((\phi_\alpha(u) \cdot s) \otimes g) = \epsilon_N(u \cdot (s \otimes g))\]
so $\epsilon_N$ is certainly a $U$-homomorphism.

Now let $m \in N$ be a weight vector, and let $a, b$ be such that $e_\alpha^a \cdot m = 0 = \overline{e}_\alpha^b \cdot m$. Let
\begin{align*}
\mathcal{I} &= \{(i, j) : i \leq 0 \mbox{ or } j \leq 0 \mbox{ or } i \leq a, j \leq b\} \subseteq \Z^2\\
m_{i, j} &= e_\alpha^{a - i} \overline{e}_\alpha^{b - j} \cdot m \mbox{ if } i \leq a, j \leq b\\
m_{i, j} &= 0 \mbox{ otherwise.}
\end{align*}

Then we can use Lemma \ref{homconstruction} to construct $\varphi \in \Hom_U(S_\alpha, \phi_\alpha^{-1}(M))$ which is a weight vector and therefore in $G_\alpha N$ such that $\varphi(f_\alpha^{-a} \fbar_\alpha^{-b}) = m$. Hence $\epsilon_N(f_\alpha^{-a} \fbar_\alpha^{-b} \otimes \varphi) = m$, so $\epsilon_N$ is surjective.

We now show $\epsilon_N$ is injective. By Lemma \ref{elementformlemma} any element of $T_\alpha G_\alpha N$ may be written as $f_\alpha^{-i} \fbar_\alpha^{-j} \otimes g$ for some $g \in G_\alpha N$. Suppose $\epsilon_N(f_\alpha^{-i} \fbar_\alpha^{-j} \otimes g) = g(f_\alpha^{-i} \fbar_\alpha^{-j}) = 0$. We may assume $g$ is weight - if not, we may write $g$ as a sum of $g_\lambda \in (G_\alpha N)^\lambda$, which by considering weight spaces must all satisfy $g_\lambda(f_\alpha^{-i} \fbar_\alpha^{-j}) = 0$, and then apply the following to each $g_\lambda$.

Let $n_{k, l} = g(f_\alpha^{-k} \fbar_\alpha^{-l})$ for all $l, k \in \Z$ and observe $n_{k, l} = 0$ if $k \leq 0$ or $l \leq 0$ or both $k \leq i$ and $l \leq j$. Choose:
\begin{align*}
\mathcal{I} &= \{(k, l) \in \Z^2 : k \leq i \mbox{ or } l \leq j\} \subseteq \Z^2 \\
m_{k, l} &= n_{k, l} \mbox{ if } k \leq i \\
m_{k, l} &= 0 \mbox{ if } l \leq j
\end{align*}
Then applying Lemma \ref{homconstruction} to this, we construct $g_1 \in \Hom_U(S_\alpha, \phi_\alpha^{-1}(M))$ which is a weight vector (and hence in $G_\alpha N$) such that:

\begin{enumerate}
\item[(a)] $g_1(f_\alpha^{-k} \fbar_\alpha^{-l}) = 0$ if $l \leq j$

\item[(b)] $g_1(f_\alpha^{-k} \fbar_\alpha^{-l}) = g(f_\alpha^{-k} \fbar_\alpha^{-l})$ if $k \leq i$
\end{enumerate}

By (a), we can define $g_1' \in G_\alpha N$ by setting $g_1'(f_\alpha^{-k} \fbar_\alpha^{-l+j}) = g_1(f_\alpha^{-k} \fbar_\alpha^{-l})$ and so $g_1 = \fbar_\alpha^j \cdot g_1'$. By (b), we can define $g_2'$ by setting $g_2'(f_\alpha^{-k+i} \fbar_\alpha^{-l}) = (g - g_1)(f_\alpha^{-k} \fbar_\alpha^{-l})$, so $(g - g_1) = f_\alpha^i \cdot g_2'$. Hence $f_\alpha^{-i} \fbar_\alpha^{-j} \otimes g = f_\alpha^{-i} \fbar_\alpha^{-j} \otimes (f_\alpha^i \cdot g_2' + \fbar_\alpha^j \cdot g_1') = 0$, so $\epsilon_N$ is injective.
\end{proof}

\begin{proof}[Proof of \ref{maintwistinglemma}(d)]

Let $N \in \calO_\epsilon^{s_\alpha(\mu)}$. Let $0 = N_0 \subseteq N_1 \subseteq \dots \subseteq N_{k-1} \subseteq N_k = N$ be a filtration of $N$ such that each section is a highest weight module. We have a short exact sequence
\[0 \rightarrow N_1 \rightarrow N \rightarrow (N/N_1) \rightarrow 0\]
and hence, since $G_\alpha$ is left exact, an exact sequence
\[0 \rightarrow G_\alpha N_1 \rightarrow G_\alpha N \rightarrow G_\alpha(N/N_1).\]
Hence $G_\alpha N$ is an extension of a submodule of $G_\alpha(N/N_1)$ by $G_\alpha N_1$. Since $G_\alpha N$ is automatically $\h$-semisimple and $\calO_\epsilon^{s_\alpha(\mu)}$ is closed under taking submodules, we can use induction on $k$ and Lemma \ref{extensionlemma} to reduce to the case where $N$ is a highest weight module of weight $(\lambda, s_\alpha(\mu))$.

Let $n$ be a highest weight generator of $N$. By Lemma \ref{homconstruction}, we can find $g_n \in G_\alpha N$ such that $g_n(f_\alpha^{-1} \fbar_\alpha^{-1}) = n$ and $g_n$ is a weight vector. We aim to show that $g_n$ is a highest weight generator of $G_\alpha N$ of weight $(s_\alpha(\lambda) - 2 \alpha, \mu)$. We now verify that $g_n$ is indeed highest weight of this weight. By Lemma \ref{weightfunctionslemma}, we certainly have that $g_n$ is a weight vector of weight $s_\alpha(\lambda) - 2\alpha$.

Let $\beta \in \Phi^+ \backslash \{\alpha\}$. Then $e_\beta \cdot g_n$ has weight $s_\alpha(\lambda) - 2\alpha + \beta$, so by Lemma \ref{weightfunctionslemma}, $(e_\beta \cdot g_n)(f_\alpha^{-i} \fbar_\alpha^{-j})$ has weight $\lambda - (i+j-2)\alpha + s_\alpha(\beta)$. Now, since $s_\alpha(\beta) \in \Phi^+ \backslash \{\alpha\}$, $\lambda - (i+j-2)\alpha + s_\alpha(\beta) \nleq \lambda$ and so $N^{\lambda - (i+j-2)\alpha + s_\alpha(\beta)} = 0$. In particular, $(e_\beta \cdot g_n)(f_\alpha^{-i} \fbar_\alpha^{-j}) = 0$ for all $i, j > 0$, so $e_\beta \cdot g_n = 0$. The same applies to $\overline{e_\beta} \cdot g_n$.

To see that $e_\alpha \cdot g_n = 0 = \overline{e_\alpha} \cdot g_n$, first consider the $((\mathfrak{sl}_2)_\alpha)_\epsilon$-module $\bigoplus_{n \in \N} N^{\lambda - n\alpha}$. Since $N$ is highest weight, this is generated as a $((\mathfrak{sl}_2)_\alpha)_\epsilon$-module by $n$, which is highest weight of weight $(\lambda(h_\alpha), s_\alpha(\mu)(h_\alpha))$. Hence since $s_\alpha(\mu)(h_\alpha) \neq 0$ and this module is clearly non-zero, so by \cite[Theorem 7.1]{W} this module is isomorphic to $M_{\lambda(h_\alpha), s_\alpha(\mu)(h_\alpha)}$, and so in particular the only elements $x \in \bigoplus_{n \in \N} N^{\lambda - n\alpha}$ such that $e_\alpha \cdot x = 0 = \overline{e}_\alpha \cdot x$ are scalar multiples of $n$. Now suppose that $e_\alpha \cdot g_n \neq 0$ or $\overline{e}_\alpha \cdot g_n \neq 0$. Then we have a non-zero element $g \in (G_\alpha N)^{s_\alpha(\lambda) - \alpha}$. By Lemma \ref{weightfunctionslemma}, $g(f_\alpha^{-1} \fbar_\alpha^{-1}) = x$ must have weight $\lambda - \alpha$. But also $e_\alpha \cdot x = 0 = \overline{e}_\alpha \cdot x$, so by the above observation $x$ is a scalar multiple of $n$, so has weight $\lambda$. Hence $x = 0$, and we can now show inductively by a similar argument that $g(f_\alpha^{-i} \fbar_\alpha^{-j}) = 0$ for all $i, j > 0$, so $g = 0$ giving a contradiction.

Finally, let $h \in \h$. Then $\overline{h} \cdot g_n \in (G_\alpha N)^{s_\alpha(\lambda) - 2\alpha}$. We have:
\begin{align*}
(\overline{h} \cdot g_n)(f_\alpha^{-1} \fbar_\alpha^{-1}) &= g_n(f_\alpha^{-1} \fbar_\alpha^{-1} \overline{h}) \\
&= g_n(\overline{h} f_\alpha^{-1} \fbar_\alpha^{-1} + \alpha(h) f_\alpha^{-2})\\
&= g_n(\overline{h} f_\alpha^{-1} \fbar_\alpha^{-1}) \\
&= \overline{h} \cdot n = \mu(h) n
\end{align*}
Now suppose $\overline{h} \cdot g_n \neq \mu(h) g_n$. Let $(i, j)$ be such that $i + j$ is minimal subject to $x:= (\overline{h} \cdot g_n - \mu(h) g_n)(f_\alpha^{-i} \fbar_\alpha^{-j}) \neq 0$, and note that $(i, j) \neq (1, 1)$ by the calculation above. Then by the minimality of $i + j$, we have $e_\alpha \cdot x = 0 = \overline{e}_\alpha \cdot x$. But by Lemma \ref{weightfunctionslemma}, we have $x \in N^{\lambda - (i+j-2)\alpha}$, which gives a contradiction by the observation in the proof that $e_\alpha \cdot g_n = 0$ that the only elements $x \in \bigoplus_{n \in \N} N^{\lambda - n\alpha}$ such that $e_\alpha \cdot x = 0 = \overline{e}_\alpha \cdot x$ are scalar multiples of $n$. Hence $\overline{h} \cdot g_n = \mu(h) g_n$, completing the proof that $g_n$ is highest weight of weight $(s_\alpha(\lambda) - 2\alpha, \mu)$ as claimed.

Now, by \ref{maintwistinglemma}(c), there is an isomorphism $\epsilon_N$ between $T_\alpha G_\alpha N$ and $N$, and furthermore this isomorphism takes $f_\alpha^{-1} \fbar_\alpha^{-1} \otimes g_n$ to $g_n(f_\alpha^{-1} \fbar_\alpha^{-1}) = n$, so $f_\alpha^{-1} \fbar_\alpha^{-1} \otimes g_n$ is a highest weight generator of $T_\alpha G_\alpha N$. Now let $L$ be the submodule of $G_\alpha N$ generated by $g_n$, and consider the inclusion $\iota: L \hookrightarrow G_\alpha N$. We wish to show that $\iota$ is an isomorphism, which will complete the proof that $N$ is highest weight and therefore the proof that $G_\alpha$ is a functor from $\calO^{s_\alpha(\mu)}$ to $\calO^\mu$. First we compute that for any $l \in L, s \in S_\alpha$:
\begin{align*}
((G_\alpha(\epsilon_N) \circ G_\alpha T_\alpha (\iota) \circ \psi_L) (l))(s) &= \epsilon_N((G_\alpha T_\alpha (\iota) \circ \psi_L) (l)(s)) \\
&= \epsilon_N(T_\alpha(\iota) \circ \psi_L(l)(s)) \\
&= \epsilon_N(T_\alpha(\iota)(s \otimes l)) \\
&= \epsilon_N(s \otimes \iota(l)) \\
&= \iota(l)(s)
\end{align*}

The first and second equalities hold since by the definition of $G_\alpha$, if $\chi \in G_\alpha T_\alpha G_\alpha N$ and $\rho$ is a map from $T_\alpha G_\alpha N$ to either $N$ or $L$ then $G_\alpha(\rho)(\chi) = \rho \circ \chi$. The final three equalities hold by the definitions of $\psi_L$, $T_\alpha(\iota)$, and $\epsilon_N$ respectively.

Hence $\iota = G_\alpha(\epsilon_N) \circ G_\alpha T_\alpha (\iota) \circ \psi_L$. Since $L$ is highest weight, by part Lemma \ref{maintwistinglemma}(b) the map $\psi_L$ is certainly an isomorphism. Similarly, by Lemma \ref{maintwistinglemma}(c), the map $\epsilon_N$ is an isomorphism, so $G_\alpha(\epsilon_N)$ is also an isomorphism. Hence to show $\iota$ is an isomorphism it suffices to show $G_\alpha T_\alpha (\iota)$ is an isomorphism, and to show this it suffices to show $T_\alpha( \iota)$ is an isomorphism. We can also conclude from this calculation that $G_\alpha T_\alpha (\iota)$ is injective.

Now we consider $T_\alpha (\iota): T_\alpha L \rightarrow T_\alpha G_\alpha N$. We have that $(T_\alpha (\iota))(f_\alpha^{-1} \fbar_\alpha^{-1} \otimes g_n) = f_\alpha^{-1} \fbar_\alpha^{-1} \otimes g_n$ which as discussed above generates $T_\alpha G_\alpha N$, so $T_\alpha(\iota)$ is certainly surjective. Let $K = \ker(T_\alpha (\iota))$ and observe since that $K$ is a submodule of $T_\alpha L \in \calO^\mu$, we have $K \in \calO^\mu$.

We have a short exact sequence:
\[0 \rightarrow K \overset{\varphi}{\rightarrow} T_\alpha L \overset{T_\alpha (\iota)}{\rightarrow} T_\alpha G_\alpha N \rightarrow 0\]
and so since $G_\alpha$ is left exact, we have an exact sequence:
\[0 \rightarrow G_\alpha K \overset{G_\alpha (\varphi)}{\rightarrow} G_\alpha T_\alpha L \overset{G_\alpha T_\alpha (\iota)}{\rightarrow} G_\alpha T_\alpha G_\alpha N\]
But $G_\alpha T_\alpha (\iota)$ is injective, so $\im(G_\alpha (\varphi)) = 0$ and hence $G_\alpha K = 0$. Now, if $K \neq 0$ then since $K \in \calO^\mu$ we can use Lemma \ref{homconstruction} to construct a non-zero element of $G_\alpha K$, so $K = 0$, i.e. $T_\alpha (\iota)$ is injective. Hence $T_\alpha (\iota)$ is an isomorphism, so by the earlier discussion $\iota$ is an isomorphism, so $G_\alpha N$ is highest weight as required.
\end{proof}

\begin{proof}[Proof of \ref{maintwistinglemma}(e)]

To show $\psi$ is natural, we must show that for any $M, N \in \calO_\epsilon^\mu$ and $f: M \rightarrow N$ that $\psi_N \circ f = G_\alpha T_\alpha f \circ \psi_M$. Again using that $G_\alpha(\rho)(\chi) = \rho \circ \chi$, we see that for any $s \in S_\alpha$ and $m \in M$:
\begin{align*}
(G_\alpha T_\alpha f \circ \psi_M(m))(s) &= G_\alpha T_\alpha f(s \otimes m) \\
&= T_\alpha f (s \otimes m) \\
&= s \otimes f(m) \\
&= \psi_M(f(m))(s) \\
&= (\psi_M \circ f)(m)(s).
\end{align*}
To show $\epsilon$ is natural, we must show that for any $M, N \in \calO_\epsilon^{s_\alpha(\mu)}$ and $f: M \rightarrow N$ that $f \circ \epsilon_M = \epsilon_N \circ T_\alpha G_\alpha f$. For any $s \in S_\alpha$ and $g \in G_\alpha M$, we have:
\begin{align*}
(\epsilon_N \circ T_\alpha G_\alpha f)(s \otimes g) &= \epsilon_N(s \otimes (G_\alpha f)(g)) \\
&= \epsilon_N(s \otimes (f \circ g)) \\
&= f(g(s)) \\
&= f(\epsilon_N(s \otimes g)) \\
&= (f \circ \epsilon_N)(s \otimes g)
\end{align*}
\end{proof}

\section{Composition Multiplicities}

\subsection{Definition of composition multiplicity}

Let $M \in \calO_\epsilon$. We define the \emph{character} of $M$ to be the function $\ch(M): \h^* \rightarrow \Z_{\geq 0}$ which sends $\lambda$ to $\dimension(M^\lambda)$, and the \emph{support} of $M$, denoted $\supp(M)$ to be the set $\supp(M)  = \{\lambda \in \h^* : M^\lambda \neq 0\}$. We define the composition multiplicities $k_\lambda(M)$ of $M$ for $\lambda \in \h^*$ using the following result, whose statement and proof is very similar to \cite[Proposition 8]{MS}:

\begin{lemma}
\label{compmultdefn}
For any $M \in \calO_\epsilon^\mu$, there are unique $k_\lambda(M) \in \Z_{\geq 0}$ such that 
\[\ch(M) = \sum k_\lambda(M) \ch(L_{\lambda, \mu})\]
\end{lemma}

\begin{proof}
We first show uniqueness. Suppose we have
\[\ch(M) = \sum a_\lambda \ch(L_{\lambda, \mu}) = \sum b_\lambda \ch(L_{\lambda, \mu})\]
for some $a_\lambda, b_\lambda \in \Z_{\geq 0}$, and $a_\lambda \neq b_\lambda$ for some $\lambda \in \h^*$. Then we let $X = \{\lambda \in \h^* : a_\lambda > b_\lambda\}$, $Y = \supp(M) \backslash X$, and
\[\chi = \sum_{\lambda \in X} (a_\lambda - b_\lambda) \ch(L_{\lambda, \mu}) = \sum_{\lambda \in Y} (b_\lambda - a_\lambda) \ch(L_{\lambda, \mu})\]
Since $a_\lambda \neq b_\lambda$ for some $\lambda$, we have that $\chi \neq 0$. Let $\nu \in \h^*$ be such that $\chi(\nu) \neq 0$ but $\chi(\nu') = 0$ for all $\nu' \geq \nu$. Now, if $\nu \in X$, then the coefficient of $\ch(L_{\nu, \mu})$ in the second sum must be 0, and since $\chi(\nu') = 0$ whenever $\nu' \geq \nu$, so must the coefficients of $L_{\nu', \mu}$ for $\nu' \geq \nu$. Hence $\chi(\nu) = 0$, giving a contradiction, and if instead $\nu \in Y$ we obtain a similar contradiction. 

To show existence of the $k_\lambda(M)$, we use induction on $n = \sum_{\lambda' \geq \lambda} \dimension(M^{\lambda'})$, which is finite by Lemma \ref{fdweightspaces}. If $n = 0$, then in particular $M^{\lambda} = 0$, so we must have $k_\lambda(M) = 0$.

Now, let $\Gamma$ be the set of non-negative integer sums of positive roots, and choose $\nu \in \lambda + \Gamma$ such that $\nu$ is maximal subject to the condition that $M^{\nu} \neq 0$. Then there must exist a highest weight vector of weight $\nu$ in $M$, generating a highest weight submodule $K$ of $M$. Consider the quotient map $M_{\nu, \mu} \rightarrow K$, and let $K' \subseteq K$ be the image of the unique maximal submodule $N_{\nu, \mu} \subseteq M_{\nu, \mu}$ under this quotient map. Both $\sum_{\lambda' \geq \lambda} \dimension(M/K)^{\lambda'}$ and $\sum_{\lambda' \geq \lambda} \dimension(K')^{\lambda'}$ are $< n$, so we have already constructed $k_\lambda(M/K)$ and $k_\lambda(K')$. We then set:
\begin{align*}
&k_\lambda(M) = k_\lambda(M/K) + k_\lambda(K') \mbox{ if } \nu \neq \lambda\\
&k_\lambda(M) = k_\lambda(M/K) + k_\lambda(K') +  1 \mbox{ if } \nu = \lambda    
\end{align*}
\end{proof}

We can also give an alternate interpretation of these which is closer to the notion of composition multiplicities for Artinian modules:

\begin{lemma}
\label{compmultalternatedefn}
Let $M \in \calO^\mu_\epsilon$ and let $M = M_0 \supseteq M_1 \supseteq M_2 \supseteq \cdots$ be a filtration of $M$ such that:
\begin{enumerate}
\item[(a)] Each $M_i/M_{i+1}$ is simple, and

\item[(b)] The intersection of all the $M_i$ is 0.
\end{enumerate}
Then $L_{\lambda, \mu}$ appears as a quotient $M_i/M_{i+1}$ precisely $k_\lambda(M)$ times. Furthermore, such a filtration always exists.
\end{lemma}

\begin{proof}
The first assertion follows from the fact that $\ch(M) = \sum \ch(M_i/M_{i+1})$ and the uniqueness of $k_\lambda(M)$. For the second, we first observe that by Theorems \ref{maintheorem} and \ref{maintwistingtheorem} we can reduce to the case $\mu = 0$.

Let $M \in \calO_\epsilon^0$, and consider the filtration $M \supseteq \overline{\g} M \supseteq \overline{\g}^2 M \supseteq \cdots$ of $M$. Each quotient $\overline{\g}^i M / \overline{\g}^{i+1} M$ lies in BGG category $\calO$, and in particular has finite length, so this filtration can be refined to one which satisfies (a). Hence it is enough to show that $\bigcap \overline{\g}^i M = 0$ for any $M \in \calO_\epsilon^0$.

Suppose $M$ is an extension of $M_1$ by $M_2$ for some $M_1, M_2 \in \calO_\epsilon^0$. The functor sending $M$ to $\overline{\g} M$ is exact, and so in particular $\dim((\overline{\g}^i M)^\lambda) = \dim((\overline{\g}^i M_1)^\lambda) + \dim((\overline{\g}^i M_2)^\lambda)$ for any $i \geq 0$ and $\lambda \in \h^*$. In particular suppose $\bigcap \overline{\g}^i M_1 = 0 = \bigcap \overline{\g}^i M_2$. Then for any $\lambda \in \h^*$, there exists $i$ such that $(\overline{\g}^i M_1)^\lambda = 0 = (\overline{\g}^i M_2)^\lambda$, so $(\overline{\g}^i M)^\lambda = 0$ and therefore $\bigcap \overline{\g}^i M = 0$. Hence by Lemma \ref{finitefiltration}, we may assume that $M$ is a highest weight module of weight $(\lambda, 0)$, i.e. a quotient of the Verma module $M_{\lambda, 0}$. Let $q: M_{\lambda, 0} \rightarrow M$ be the quotient map. We have $\overline{\g}^i M = q(\overline{\g}^i M_{\lambda, 0})$, so by considering weight spaces it is in fact enough to show that $\bigcap \overline{\g}^i M_{\lambda, 0} = 0$.

For $\mathbf{n}, \mathbf{m} \in \Z_{\geq 0}^{\Phi^+}$, let $u(\mathbf{n}, \mathbf{m}) = \Pi f_{\alpha_k}^{n_k} \Pi \overline{f}_{\alpha_k}^{m_k} \in U(\g_\epsilon)$. We aim to show that $\overline{\g}^i M_{\lambda, 0} = \vspan\{u(\mathbf{n}, \mathbf{m}) \otimes 1_{\lambda, 0}: |\mathbf{m}| \geq i\}$, which then implies the desired result. Let $x \in \g$ and consider $\overline{x} \cdot u(\mathbf{n}, \mathbf{m}) \otimes 1_{\lambda, 0} = [\overline{x}, u(\mathbf{n}, \mathbf{0})] u(\mathbf{0}, \mathbf{m}) \otimes 1_{\lambda, 0} + u(\mathbf{n}, \mathbf{0}) \, \overline{x} \, u(\mathbf{0}, \mathbf{m}) \otimes 1_{\lambda, 0}$. Now, $u(\mathbf{n}, \mathbf{0}) \, \overline{x} \, u(\mathbf{0}, \mathbf{m}) \otimes 1_{\lambda, 0}$ is clearly a sum of terms of the form $u(\mathbf{n}, \mathbf{m'}) \otimes 1_{\lambda, 0}$, where $\mathbf{m'}$ is such that $|\mathbf{m'}| = |\mathbf{m}| + 1$. To deal with the other term, we note that for some $\alpha \in \Phi^+$ and $\mathbf{n'} \in \Z_{\geq 0}^{\Phi^+}$ such that $|\mathbf{n'}| = |\mathbf{n}| - 1$, we have that $[\overline{x}, u(\mathbf{n}, \mathbf{0})]u(\mathbf{0}, \mathbf{m}) \otimes 1_{\lambda, 0} = f_\alpha \overline{x} \, u(\mathbf{n'}, \mathbf{m}) \otimes 1 + [\overline{x}, f_\alpha] u(\mathbf{n'}, \mathbf{m}) \otimes 1_{\lambda, 0}$. Since $[\overline{x}, f_\alpha] = \overline{\{x, f_\alpha\}} \in \overline{\g}$, by induction on $|\mathbf{n}|$ we have that this is also a sum of terms of the form $u(\mathbf{n''}, \mathbf{m'}) \otimes 1_{\lambda, 0}$ for some $\mathbf{n''}, \mathbf{m'} \in \Z_{\geq 0}^{\Phi^+}$ where $|\mathbf{m'}| = |\mathbf{m}| + 1$. Hence applying any element of $\overline{\g}$ to $u(\mathbf{n}, \mathbf{m}) \otimes 1_{\lambda, 0}$ gives a sum of terms of of the form $u(\mathbf{n'}, \mathbf{m'}) \otimes 1$ where $|\mathbf{m'}| = |\mathbf{m}| + 1$, from which the result follows by induction.
\end{proof}

In light of this result, from now on we use the notation $[M : L_{\lambda, \mu}]$ in place of $k_\lambda(M)$.

\begin{corollary}
\label{compmultequivalence}
The parabolic induction and restriction functors $I, R$ and the twisting functors $T_\alpha, G_\alpha$ preserve composition multiplicities, i.e. for $F = I, R, T_\alpha$ or $G_\alpha$, any suitable module $M$, and suitable $\lambda, \mu \in \h^*$, we have $[M : L_{\lambda, \mu}] = [F(M) : F(L_{\lambda, \mu})]$.
\end{corollary}

\begin{proof}
Let $M = M_0 \supseteq M_1 \supseteq M_2 \supseteq \cdots$ be a filtration of $M$ of the form described in Lemma \ref{compmultalternatedefn}. Then $F(M) = F(M_0) \supseteq F(M_1) \supseteq F(M_2) \supseteq \cdots$ is also a filtration of this form, and for all $i$ we have $F(M_i)/F(M_{i+1}) \cong F(M_i/M_{i+1})$, so $[M: L_{\lambda, \mu}] = [F(M): F(L_{\lambda, \mu})]$.
\end{proof}

\subsection{Computation of composition multiplicities for Verma modules}

We now wish to compute the composition multiplicities $[M : L_{\lambda, \mu}]$ in the case where $M = M_{\lambda, 0}$, which we will then use to compute the multiplicities for all Verma modules $M_{\lambda, \mu}$. As before, for $\mathbf{n}, \mathbf{m} \in \Z_{\geq 0}^{\Phi^+}$ we let $u(\mathbf{n}, \mathbf{m}) = \Pi f_{\alpha_k}^{n_k} \Pi \overline{f}_{\alpha_k}^{m_k} \in U(\g_\epsilon)$. We then have:

\begin{lemma}
Let $M = M_{\lambda, 0}$ for some $\lambda \in \h^*$. Then
\[M = \vspan\{u(\mathbf{n}, \mathbf{m}) \otimes 1_{\lambda, 0}: \mathbf{n}, \mathbf{m} \in \Z_{\geq 0}^{\Phi^+}\} = \bigoplus_{\mathbf{m} \in \Z_{\geq 0}^{\Phi^+}} (\vspan\{u(\mathbf{n}, \mathbf{m}) \otimes 1_{\lambda, 0}: \mathbf{n} \in \Z_{\geq 0}^{\Phi^+}\})\]
Furthermore, each of these summands is an $\h$-module, and as $\h$-modules each summand is isomorphic to the Verma module $M_{\lambda - \sum m_k \alpha_k}$ for $\g$. Hence $\ch M = \sum_{\mathbf{m} \in \Z_{\geq 0}^{\Phi^+}} \ch M_{\lambda - \sum m_k \alpha_k}$.
\end{lemma}

\begin{proof}
The equality $M = \bigoplus_{\mathbf{m} \in \Z_{\geq 0}^{\Phi^+}} (\vspan\{u(\mathbf{n}, \mathbf{m}) \otimes 1_{\lambda, 0}: \mathbf{n} \in \Z_{\geq 0}^{\Phi^+}\})$ holds by the definition of $M_{\lambda, \mu}$, and since $u(\mathbf{n}, \mathbf{m})$ is a weight vector for each $\mathbf{n}, \mathbf{m} \in \Z_{\geq 0}^{\Phi^+}$, each summand is an $\h$-submodule of $M$. We also observe that there is an isomorphism between $\vspan\{u(\mathbf{n}, \mathbf{m}) \otimes 1_{\lambda, 0}: \mathbf{n} \in \Z_{\geq 0}^{\Phi^+}\}$ and $M_{\lambda - \sum m_k \alpha_k}$ given by sending to $u(\mathbf{n}, \mathbf{m}) \otimes 1_{\lambda, 0}$ to $\Pi f_{\alpha_k}^{n_k} \otimes 1_{\lambda}$ for $\mathbf{n} \in \Z_{\geq 0 }^{\Phi^+}$.
\end{proof}

\begin{corollary}
\label{zeroblockcompmult}
$[M_{\lambda, 0}: L_{\lambda', 0}] = \sum_{\mathbf{m} \in \Z_{\geq 0}^{\Phi^+}} [M_{\lambda - \sum m_k \alpha_k}(\g): L_{\lambda'}(\g)] = \sum_{\chi \in \mathbb{Z}\Phi} p(\chi) [M_{\lambda + \chi}(\g): L_{\lambda'}(\g)]$, where $p$ is Konstant's partition function and $M_{\lambda + \chi}(\g)$ and $L_{\lambda'}(\g)$ are respectively the Verma module of weight $\lambda + \chi$ for $\g$ and the simple module of weight $\lambda'$ for $\g$.
\end{corollary}

\begin{proof}
By Lemma \ref{simple0modules}, $\ch L_{\lambda, 0} = \ch L_\lambda(\g)$ for all $\lambda \in \h^*$, so this follows from the previous lemma.   
\end{proof}

We now observe that our equivalence using the parabolic induction functor takes highest weight modules of weight $(\lambda, 0)$ in $\calO_\epsilon^0(\g^\mu)$ to highest weight modules of weight $(\lambda, \mu)$ in $\calO_\epsilon^\mu(\g)$. Hence it takes $M_{\lambda, 0} \in \calO_\epsilon^0(\g^\mu)$ to $M_{\lambda, \mu} \in \calO_\epsilon^\mu(\g)$, and similarly takes $L_{\lambda, 0} \in \calO_\epsilon^0(\g^\mu)$ to $L_{\lambda, \mu} \in \calO_\epsilon^\mu(\g)$, so by Corollaries \ref{compmultequivalence} and \ref{zeroblockcompmult} we have:

\begin{corollary}
$[M_{\lambda, \mu}: L_{\lambda', \mu}] = \sum_{\chi \in \mathbb{Z}\Phi} p(\chi) [M_{\lambda + \chi}(\g^\mu): L_{\lambda'}(\g^\mu)]$, where $M_{\lambda + \chi}(\g^\mu)$ and $L_{\lambda'}(\g^\mu)$ are the Verma module and the simple module for $\g^\mu$ of the appropriate weights.
\end{corollary}

We now define an action $\bullet_2$ of $W$ on $\h^*$ by $w \bullet_2 \lambda := w(\lambda + 2\rho) - 2\rho$ where $\rho$ is half the sum of the positive roots. Note this is similar to the normal dot action of $W$, but with a shift of $2\rho$ rather than $\rho$. We observe that if $\alpha$ is any simple root, then since $s_\alpha(\alpha) = -\alpha$ and $s_\alpha$ permutes the other positive roots we have $s_\alpha(\rho) = \rho - \alpha$. We then have:
\begin{align*}
s_\alpha \bullet_2 \lambda &= s_\alpha(\lambda + 2\rho) - 2\rho \\
&= s_\alpha(\lambda) + 2\rho - 2\alpha - 2\rho \\
&= s_\alpha(\lambda) - 2\alpha
\end{align*}
Hence by our calculations in the proof of Lemma \ref{maintwistinglemma}(a), the twisting functors $T_\alpha$ take highest weight modules of weight $(\lambda, \mu)$ to highest weight modules of weight $(s_\alpha \bullet_2 \lambda, s_\alpha(\mu))$ and hence take $M_{\lambda, \mu}$ to $M_{s_\alpha \bullet_2 \lambda, s_\alpha(\mu)}$ and $L_{\lambda, \mu}$ to $L_{s_\alpha \bullet_2 \lambda, s_\alpha(\mu)}$. Applying Corollary \ref{compmultequivalence} again, we obtain:

\begin{corollary}
Let $\lambda, \lambda', \mu \in \h^*$, and let $w \in W$ be of the form $w = s_{\alpha_n} s_{\alpha_{n-1}} \cdots s_{\alpha_1}$ for some simple reflections $s_{\alpha_i}$ such that:

(a) $\g^{w(\mu)} = \mathfrak{l}$ for some Levi factor $\mathfrak{l}$ of a parabolic $\p$ and,

(b) For each $1 \leq i \leq n$, we have $((s_{\alpha_{i-1}} \cdots s_{\alpha_1})\mu)(h_{\alpha_i}) \neq 0$.

Then:
\begin{align*}
[M_{\lambda, \mu}: L_{\lambda', \mu'}] = [M_{w \bullet_2 \lambda, w(\mu)}: L_{w \bullet_2 \lambda', w(\mu')}] = \delta_{\mu, \mu'}\sum_{\chi \in \mathbb{Z}\Phi} p(\chi) [M_{w \bullet_2 \lambda + \chi}(\g^{w(\mu)}): L_{w \bullet_2 \lambda'}(\g^{w(\mu)})]
\end{align*}
where $M_{w \bullet_2 \lambda + \chi}(\g^{w(\mu)})$ and $L_{w \bullet_2 \lambda'}(\g^{w(\mu)})$ are the Verma module and the simple module for $\g^{w(\mu)}$ of the appropriate weights.
\end{corollary}

\end{document}